\documentclass[10pt]{amsart}
\usepackage{graphicx, amsmath, fullpage, amssymb, amsthm, amsfonts, color, algorithm}
\usepackage{enumerate}
\newtheorem{theorem}{Theorem}[section]
\newtheorem{lemma}[theorem]{Lemma}

\theoremstyle{definition}
\newtheorem{definition}[theorem]{Definition}

\newtheorem{corollary}[theorem]{Corollary}
\usepackage[toc,page]{appendix} 

\theoremstyle{remark}
\newtheorem{remark}[theorem]{Remark}
\numberwithin{equation}{section}
\usepackage[colorlinks,
            linkcolor=red,
            anchorcolor=red,
            citecolor=blue
            ]{hyperref}
 
\newcommand{\N}{\mathcal N}

\newcommand{\codeg}{\textnormal{codeg}}

\renewcommand{\epsilon}{\varepsilon}

\begin{document}

\title{On the Second Eigenvalue of Random Bipartite Biregular Graphs}

\author{Yizhe Zhu}
\address{Department of Mathematics, University of California Irvine, Irvine, CA 92697}
\email{yizhe.zhu@uci.edu}
\thanks{}

\subjclass[2000]{Primary 60C05, 60B20; Secondary 05C50}
\keywords{random bipartite biregular graph, spectral gap, switching, size biased coupling}
\date{\today}

\begin{abstract}

We  consider the spectral gap of a uniformly chosen random  $(d_1,d_2)$-biregular bipartite graph $G$ with $|V_1|=n, |V_2|=m$, where  $d_1,d_2$ could possibly grow with $n$ and $m$. Let $A$ be the adjacency matrix of $G$. Under the assumption that $d_1\geq d_2$ and $d_2=O(n^{2/3}),$ we show that $\lambda_2(A)=O(\sqrt{d_1})$ with high probability. As a corollary, combining the results from \cite{tikhomirov2019spectral}, we  showed that the second singular value of a uniform random $d$-regular digraph is $O(\sqrt{d})$ for $1\leq d\leq n/2$ with high probability.  Assuming $d_2$ is fixed and $d_1=O(n^2)$, we further prove that for a random $(d_1,d_2)$-biregular bipartite graph, $|\lambda_i^2(A)-d_1|=O(\sqrt{d_1})$ for all $2\leq i\leq n+m-1$ with high probability. 
The proofs of the two results are based on the size biased coupling method introduced in \cite{cook2018size} for random $d$-regular graphs and several new switching operations we defined for random bipartite biregular graphs.
\end{abstract}

\maketitle

\section{Introduction}

An expander graph is a sparse graph that has strong connectivity properties and exhibits rapid mixing. Expander graphs
play an important role in computer science, including sampling, complexity theory, and the design of
error-correcting codes (see \cite{hoory2006expander,alon1985lambda1}). When a graph is $d$-regular, i.e., each vertex has degree $d$, quantification  of expansion is possible based on the eigenvalues of the adjacency matrix. Let $A$ be the adjacency matrix of a $d$-regular graph. The first eigenvalue $\lambda_1(A)$ is always $d$. The second eigenvalue in absolute value $\lambda(A)=\max\{ \lambda_2(A),-\lambda_n(A)\}$ is of particular interest, since the difference between $d$ and $\lambda$, also known as the \textit{spectral gap}, provides an estimate on the expansion property of the graph. The study of the spectral gap in $d$-regular graphs with fixed $d$ had the first breakthrough in the
Alon-Boppana bound. It was proved in \cite{alon1986eigenvalues,nilli1991second}  that for $d$-regular graph $\lambda(A)\geq 2\sqrt{d-1}-o(1).$ Regular graphs with $\lambda(A)\leq 2\sqrt{d-1}$  are called \textit{Ramanujan}.
In \cite{friedman2008proof}, Friedman proved Alon's conjecture in \cite{alon1986eigenvalues} that for the uniform model of random $d$-regular graph with fixed $d\geq 3$,  $\lambda(A)\leq 2\sqrt{d-1}+\epsilon$ asymptotically almost surely for any $\varepsilon>0$. The result implies almost all random regular graphs are nearly Ramanujan.  Bordenave \cite{bordenave2019new} gave a simpler proof that $\lambda(A)\leq 2\sqrt{d-1}+\epsilon_n$ for a sequence $\epsilon_n\to 0$ asymptotically almost surely.  Very recently, in  \cite{huang2021spectrum}, this estimate was improved  to $\lambda(A)\leq 2\sqrt{d-1}+O(n^{-c})$.

A generalization of Alon's question is to consider the spectral gap of random $d$-regular graphs when $d$ grows with $n$. In \cite{broder1998optimal} the authors showed that for $d=o(n^{1/2})$, a uniformly distributed random $d$-regular graph satisfies $\lambda(A)=O(\sqrt{d})$ with high probability.  The authors worked with random regular multigraphs
drawn from the configuration model and translated the result for the uniform model by the contiguity argument,  which hit a barrier at $d=o(n^{1/2})$.  The range of $d$ for the bound $\lambda(A)=O(\sqrt{d})$ was extended to $d=O(n^{2/3})$ in \cite{cook2018size} by proving concentration  results directly for the uniform model. In \cite{tikhomirov2019spectral} it was proved that the $O(\sqrt{d})$ bound holds for $n^{\epsilon}\leq d\leq n/2$  with $\epsilon\in (0,1)$.  Vu in \cite{vu2008random,vu2014combinatorial} conjectured   that $\lambda(A)=(2+o(1))\sqrt{d(1-d/n)}$ with high probability when $d\leq n/2$ and $d$ tends to infinity with $n$. The combination of the results in \cite{cook2018size} and \cite{tikhomirov2019spectral} confirms Vu's conjecture up to a multiplicative constant. Recently, the authors
in \cite{bauerschmidt2020edge} showed that for $ n^{\epsilon}\leq d \leq n^{2/3-\epsilon}$ with any $\epsilon>0$, $\lambda(A)=(2+o(1))\sqrt{d-1}$ with high probability. Later, it was proved in \cite{sarid2022spectral} that Vu's conjecture holds for $\log^{10}(n) \ll d\leq cn$, where $c$ is a small constant. In \cite{he2022spectral}, the same bound was proved for $n^{2/3}\ll d\leq n/2$, which settled this conjecture in the regime $\log^{10} n\ll d\leq n/2$.

\subsection{Random bipartite biregular graphs}
In many applications, one would like to construct bipartite expander graphs with two unbalanced disjoint vertex sets, among which bipartite biregular graphs are of particular interest.
An \textit{$(n,m,d_1,d_2)$-bipartite biregular  graph} is a bipartite graph $G=(V_1,V_2, E)$ where $|V_1|=n, |V_2|=m$ and every vertex in $V_1$ has degree $d_1$ and every vertex in $V_2$ has degree $d_2$. Note that we must have $nd_1=md_2= |E|$. When the number of vertices is clear, we call it a $(d_1,d_2)$-biregular bipartite graph for simplicity.   Let $X\in \{0,1\}^{n\times m}$ be a matrix indexed by $V_1\times V_2$ such that 
$X_{ij}=1$ if and only if $ (i,j)\in E$. The adjacency matrix of a $(d_1,d_2)$-biregular bipartite graph with $V_1=[n], V_2=[m]$ can be written as 
\begin{align}\label{eq:A}
    A=\begin{bmatrix}
    0  &X\\
    X^{\top} &0 
    \end{bmatrix}.
\end{align}
All eigenvalues of $A$ come in pairs as $\{-\lambda, \lambda\}$, where $|\lambda|$ is a singular value of $X$ along with at least $|n-m|$ zero eigenvalues. It's easy to see  $\lambda_1(A)=-\lambda_{n+m}(A)=\sqrt{d_1d_2}$. The difference between $\sqrt{d_1d_2}$ and $\lambda_2(A)$ is called the spectral gap for the bipartite biregular graph. The spectral gap of  bipartite biregular graphs has found applications in  error correcting codes, matrix completion and community detection, see for example \cite{tanner1981recursive,sipser1996expander, gamarnik2017matrix,brito2018spectral,burnwal2020deterministic}.    Previous works of \cite{feng1996spectra,sole1996spectra} showed an analog of Alon-Boppana bound for bipartite biregular graph: for any sequence of $(d_1,d_2)$-biregular bipartite graphs with fixed $d_1$ and $d_2$, as the number of vertices tends to infinity, for any $\epsilon>0$, 
$\liminf_{n\to\infty}\lambda_2\geq \sqrt{d_1-1}+\sqrt{d_2-1}-\epsilon.$ 
In \cite{feng1996spectra}, a $(d_1,d_2)$-biregular bipartite graph is defined to be Ramanujan if $\lambda_2\leq \sqrt{d_1-1}+\sqrt{d_2-1}$.
It was shown in \cite{marcus2015interlacing} that there exist infinite families of $(d_1,d_2)$-biregular bipartite Ramanujan graphs for every $d_1,d_2\geq 3$. Very recently, for fixed $d_1$ and $d_2$, \cite{brito2018spectral} showed that almost all $(d_1,d_2)$-biregular bipartite graphs are almost Ramanujan in the  sense that
$
 \lambda_2\leq \sqrt{d_1-1}+\sqrt{d_2-1}+\epsilon_n   
$ for a sequence $\epsilon_n\to 0$ asymptotically almost surely.

In this paper, we  consider the spectral gap of a uniformly chosen random  $(d_1,d_2)$-biregular bipartite graph with $V_1=[n], V_2=[m]$, where $d_1,d_2$ can possibly grow with $n$ and $m$.
Without loss of generality, we assume $d_1\geq d_2$. Let $X\in \mathbb R^{n\times m}$ be the \textit{biadjacency matrix} of a bipartite biregular graph. Namely, $X_{uv}=1$ if $(u,v)\in E, u\in V_1$ and $v\in V_2$. Let $\mathbf{1}_m=(1,\dots,1)^{\top} \in \mathbb R^m$. Since $X^{\top}X\mathbf{1}_{m}=d_1d_2\mathbf{1}_m $, the largest singular value satisfies $\sigma_1(X)=\sqrt{d_1d_2}.$ Moreover, the second eigenvalue of $A$ in \eqref{eq:A} is equal to  $\sigma_2(X)$. 

Our first result is the second eigenvalue bound, which is an extension of \cite{brito2018spectral} to the case where $d_1,d_2$ can possibly grow with $n,m$.

\begin{theorem}\label{thm:main1} Let $A$ be the adjacency matrix of a uniform random $(n,m,d_1,d_2)-$bipartite biregular graph with $d_1\geq d_2$.
For any $K>0$, if $d_2\leq \frac{1}{2}n^{2/3}$, then there exists  a constant $\alpha>0$ depending only on $K$ such that 
\begin{align}\label{eq:thm111}
    \mathbb P\left(\lambda_2(A)\leq \alpha \sqrt{d_1}\right)\geq 1-m^{-K}-e^{-m}.
\end{align}
\end{theorem}

A \textit{$d$-regular digraph} 
on $n$  vertices is a digraph with each vertex having $d$ in-neighbors and $d$ out-neighbors. We can interpret the biadjacency matrix $X$ of a random $(n,n,d,d)$-bipartite biregular graph as the adjacency matrix of a random $d$-regular digraph. Therefore the following corollary holds.
\begin{corollary}\label{cor:cor1}
 Let $A$ be the adjacency matrix of a uniform random $d$-regular digraph on $n$ vertices. For any  $ K>0$, if  $d\leq\frac{1}{2} n^{2/3}$, then there exists a constant $\alpha>0$ depending only on $K$ such that 
\begin{align}
    \mathbb P\left(\sigma_2(A)\leq \alpha \sqrt{d}\right)\geq 1-n^{-K}-e^{-n}.
\end{align}
\end{corollary}
For fixed $d$, it was proved in \cite{brito2018spectral} that $\sigma_2(A)\leq 2\sqrt{d-1}+o(1)$ with high probability. 
In \cite{tikhomirov2019spectral} the authors showed $\sigma_2(A)=O(\sqrt{d})$ with high probability when $n^{\epsilon}\leq d\leq \frac{n}{2}$. Combining Corollary \ref{cor:cor1} and their result, we confirm a conjecture in \cite{cook2017discrepancy} that a uniform $d$-regular digraph has  $\sigma_2(A)=O(\sqrt{d})$ for $1\leq d\leq n/2$ with high probability.  Although it  was not stated in \cite{cook2017discrepancy}, analogous to Vu's conjecture \cite{vu2008random,vu2014combinatorial}, a more precise version of the conjecture  can be formulated as  $\sigma_2(A)\leq (2+o(1))\sqrt{d-1}$ for $1\ll  d\leq n/2$.

Order the eigenvalue of  $A$ in modulus as $|\lambda_1(A)|\geq |\lambda_2(A)|\geq \cdots \geq |\lambda_n(A)|$.
It is also shown in \cite{coste2017spectral} that for random $d$-regular digraph,  $|\lambda_2(A)|\leq \sqrt{d}+\epsilon$ with high probability for any $\epsilon>0$. Note that for  a $d$-regular digraph, $\lambda_1(A)=\sigma_1(A)=d$.  By  Weyl's inequality between eigenvalues and singular values \cite[Theorem 3.3.2]{horn_johnson_1991}), we have 
$
   |\lambda_1(A)\lambda_2(A)|\leq \sigma_1(A)\sigma_2(A),
$
and $|\lambda_2(A)|\leq \sigma_2(A)$  for any $d$-regular digraph. Therefore, Corollary \ref{cor:cor1} and  \cite[Theorem B]{tikhomirov2019spectral} together also imply $|\lambda_2(A)|=O(\sqrt{d})$ with high probability. Very recently,  for the permutation model of random regular digraphs, it was proved in \cite{coste2022characteristic} that $|\lambda_2(A)|\leq (1+\varepsilon)\sqrt{d}$ when $1\leq d\leq n^{o(1)}$.

Now we assume $d_2$ is a bounded constant and $d_1$ could  grow with $n$, we refine our estimate  as follows. 
\begin{theorem}\label{thm:main2} Let $A$ be the adjacency matrix of a random $(n,m,d_1,d_2)$-bipartite biregular graph with $d_1\geq d_2$. Let $d_2$ be a fixed constant independent of $n$, and $n\geq 4d_2$.
For any constants $K, C_1>0$, if $d_1\leq C_1n^2$, there exists $\alpha>0$ depending only on $C_1, d_2,  K$ such that 
\begin{align}\label{eq:nontrivial}
    \mathbb P\left(\max_{2\leq i\leq m+n-1}|\lambda_i^2 (A)-d_1|\leq \alpha \sqrt{d_1}\right)\geq 1-n^{-K}-e^{-n}.
\end{align}
\end{theorem}

We see from \eqref{eq:nontrivial} that the absolute values of all nontrivial eigenvalues of $A$ are concentrated around $ \sqrt{d_1}$, which is an improved estimate compared to Theorem \ref{thm:main1}. Theorem \ref{thm:main2} is also used in \cite{dumitriu2020global} to study the global eigenvalue fluctuation of random bipartite biregular graphs. We believe the conditions on $d_1,d_2$ are only  technical assumptions due to the limitation of the method we used. It's possible that by other methods one can extend the range of $d_1,d_2$ such that \eqref{eq:nontrivial} still holds.

\section{Size biased coupling}\label{sec:Sizecoupling}

A key ingredient in the proof of our main results  is to show the concentration of linear functions for  a random matrix. In this section, we collect some useful results on size biased coupling. For more background on size biased coupling, see the survey \cite{arratia2019size} and Section 3 in \cite{cook2018size}.

Let $X$ be a nonnegative random variable with  $\mu=\mathbb EX>0$. We say \textit{$X^s$ has the $X$-size biased distribution} if 
$
    \mathbb E[Xf(X)]=\mu \mathbb E[f(X^s)]
$
for all functions $f$ such that the left hand side above exists. We say a pair of random variable $(X, X^s)$ defined on a common probability space is a \textit{size biased coupling} for $X$ when $X^s$ has the $X$-size biased distribution. The following lemma for the sum of indicator random variables will be convenient in our setting. 

\begin{lemma}[Lemma 3.1 in \cite{cook2018size}]\label{lem:biasindicator}
Let $X_i=a_iF_i$ where $F_i$ is a non-constant random variable taking values in $\{0,1\}$ and $a_i\geq 0$. Let $X=\sum_{i=1}^n X_i.$ Let $(X_1^{(i)},\dots, X_n^{(i)})$ be random variables such that $X_i^{(i)}=a_i$ and $(X_j^{(i)})_{j\not=i}$ are  distributed as $(X_j^{(i)})_{j\not=i}$ conditioned on $F_i=1$. Independent of everything else, choose a random index $I$ such that $\mathbb P(I=i)=\mathbb EX_i/\mathbb EX$ for $1\leq i\leq n$. Then $X^{s}=\sum_{i=1}^n X_i^{(I)}$ has the size biased distribution of $X$.
\end{lemma}

The following result in \cite{cook2018size}  provides concentration inequalities from the construction of size biased couplings. For $x\in \mathbb R$, define $(x)_+:=\max\{x,0\}$.

\begin{lemma}[Theorem 3.4. in \cite{cook2018size}]\label{thm:tailestimate1}
 Let $(X,X^s)$ be a size biased coupling with $\mathbb EX=\mu$, $\mathcal B$ be an event on which $X^s-X\leq c$. Let $D=(X^s-X)_+$ and suppose $\mathbb E[D\mathbf{1}_{\mathcal B}\mid X]\leq \tau^2/\mu$ almost surely. Define $
     h(x)=(1+x)\log(1+x)-x ,\quad x\geq -1.
 $
 Then the following holds:
 \begin{enumerate}
     \item If $\mathbb P[\mathcal B\mid X^s]\geq p$ almost surely, then for $x\geq 0$,
     $
         \mathbb P \left(X-\frac{\mu}{p}\geq x\right)\leq \exp\left( -\frac{\tau^2}{pc^2} h\left(\frac{pcx}{\tau^2} \right)\right).
    $
    \item If $\mathbb P[\mathcal B\mid X]\geq p$ almost surely, then for $x\geq 0$,
     $
         \mathbb P (X-p\mu\leq  -x)\leq \exp\left( -\frac{\tau^2}{c^2} h\left(\frac{cx}{\tau^2} \right)\right).
     $
 \end{enumerate}
\end{lemma}

\section{Switching for bipartite biregular graphs}\label{sec:switching}

The method of switchings, developed by McKay and Wormald \cite{mckay1981expected,wormald1999models}, has been used to approximately enumerate regular graphs, counting subgraphs in random regular graphs, see for example \cite{mckay1981subgraphs,mckay2004short,kim2007small,krivelevich2001random2}.  In recent years, combined with other random matrix  techniques, switching has become a useful tool to study the spectra of random regular graphs  \cite{johnson2015exchangeable, cook2018size,bauerschmidt2017local,bauerschmidt2017bulk,bauerschmidt2019local,bauerschmidt2020edge}. It was also applied to spectral analysis of other random graph models including  random regular digraphs \cite{cook2017discrepancy,cook2017singularity,cook2019circular,litvak2019structure,litvak2017adjacency,litvak2019smallest,litvak2018circular} and random bipartite biregular graphs \cite{yang2017local,yang2017bulk}. 
The switching operation defines a natural Markov chain called the “switch chain”, which is often used to sample random graphs. In \cite{tikhomirov2020sharp}, the authors derived Poincar\'e inequalities for the switch chain on $d$-regular bipartite graphs when $3\leq d\leq cn$. It is known that such functional inequalities imply corresponding concentration inequalities \cite{ledoux2001concentration}, which provides a possible approach to study the spectral gap of uniform random regular bipartite graphs.

In this section, we introduce the switching operations on bipartite biregular graphs, which are different from \cite{yang2017local} and  involve more vertices.  Our definition of switchings is an analog of  the ``double switchings" defined for regular graphs in Section 4 of \cite{cook2018size}, and is suitable for bipartite biregular graphs.  
The switchings will be used to construct a  coupling between $X$ and $X^{(u_1v_1)}$, where $X$ is the biadjacency matrix of a random bipartite biregular graph and $X^{(u_1v_1)}$ is the distribution of $X$ conditioned on $X_{u_1v_1}=1$.

\begin{definition}[valid switchings]\label{def:switch} Assume $X_{u_1v_2}=X_{u_2v_1}=X_{u_3v_3}=1$ and $X_{u_1v_1}=X_{u_2v_3}=X_{u_3v_2}=0$. We define  $(u_1,u_2,u_3,v_1,v_2,v_3)$ to be a \textit{valid forward switching} for $X$ as follows. After switching in the graph $G$, the edges $u_1v_1,u_2v_3,u_3v_2$ are added and the edges $u_1v_2,u_2v_1,u_3v_3$ are removed.  
In a similar way, suppose   $X_{u_1v_2}=X_{u_2v_1}=X_{u_3v_3}=0$ and $X_{u_1v_1}=X_{u_2v_3}=X_{u_3v_2}=1$, define  $(u_1,u_2,u_3,v_1,v_2,v_3)$ to be a \textit{valid backward switching} for $X$ if  after switching, the edges $u_1v_2,u_2v_1,u_3v_3$  are added and the edges  $u_1v_1,u_2v_3,u_3v_2$ are removed.  In both of the forward and backward switchings, we assume $u_1,u_2,u_3,v_1,v_2,v_3$ are distinct vertices. 
\end{definition}

See Figure \ref{fig:my_label1} for an example for a valid forward switching. By reversing the arrow in Figure \ref{fig:my_label1}, we obtain a valid backward switching from the right to the left.
 The following lemma estimates the number of valid forward and backward switchings.
\begin{figure}[ht]
    \centering
    \includegraphics[width=0.5\linewidth]{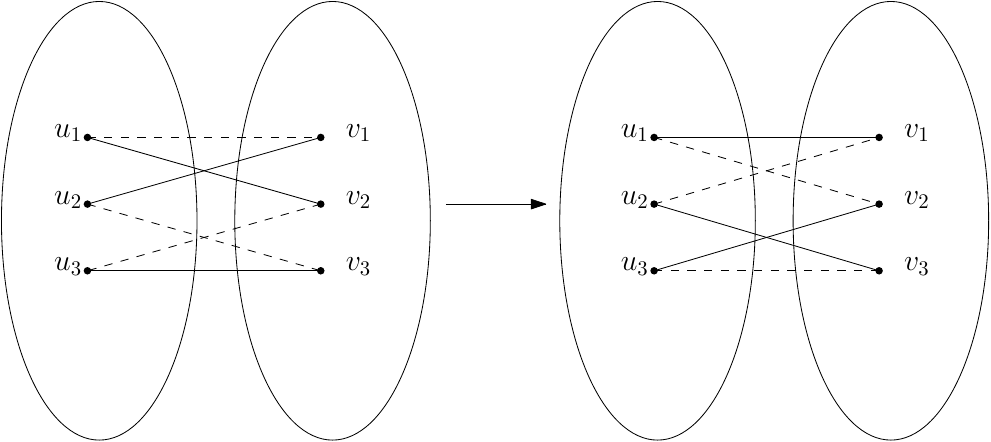}
    \caption{a valid forward switching}
    \label{fig:my_label1}
\end{figure}

\begin{lemma} \label{lem:switchingcounts}
 Let $s_{u_1v_1}(G)$ and $t_{u_1u_1}(G)$ be the number of valid  forward and backward switchings of the form $(u_1,\cdot,\cdot,v_1,\cdot,\cdot)$, respectively. Then the following inequalities hold:
\begin{enumerate}
    \item If $X_{u_1v_1}=0$, 
    $
        d_1^2d_2(n-2d_2)\leq s_{u_1v_1}(G)\leq d_1^2d_2(n-d_2).
    $
    \item If $X_{u_1v_1}=1$, 
    $
        d_1^2(n-d_2)(n-2d_2)\leq t_{u_1v_1}(G)\leq d_1^2(n-d_2)^2.
    $
\end{enumerate}
\end{lemma}
\begin{proof}
Let $\N(v)$ be the set of neighborhood of  a vertex $v$ in the graph $G$. Define $\overline{\N}(v)=V_2 \setminus \N(v)$ if $v\in V_1$, and $\overline{\N}(v)=V_1\setminus \N(v)$ if $v\in V_2$.   Fix $u_1,v_1$ and assume $X_{u_1v_1}=0$. By choosing $u_2\in\N(v_1),v_2\in \N(u_1), u_3\in \overline{\N}(v_2),v_3\in \N(u_3),$ we have $d_1^2d_2(n-d_2)$ many tuples, which gives the upper bound on $s_{u_1v_1}(G)$. Among those $d_1^2d_2(n-d_2)$ many tuples, a tuple is a valid forward switching if and only if $v_3\in\overline{\N}(u_2)$. By choosing $u_2\in\N(v_1),v_2\in\N(u_1), v_3\in \N(u_2),u_3\in \N(v_3)$, we have at most $d_1^2d_2^2$ many tuples that are not valid forward switchings. Therefore
$ s_{u_1v_1}(G)\geq d_1^2d_2(n-d_2)-d_1^2d_2^2=d_1^2d_2(n-2d_2).$
This completes the proof for the first claim.
Now assume $X_{u_1v_1}=1$. By choosing $u_2\in\overline{\N}(v_1),v_3\in \N(u_2),u_3\in\overline{\N}(v_3),v_2\in \N(u_3)$, we have $d_1^2(n-d_2)^2$ many tuples, giving the upper bound of $t_{u_1v_1}(G)$. Among those tuples, a tuple is not a valid backward switching if and only if $v_2\in \N(u_1)$. We can  bound the number of invalid tuples by choosing $v_2\in \N(u_1), u_2\in \overline{\N}(v_1), v_3\in \N(u_2), u_3\in \N(v_2)$, which has at most  $d_1^2d_2(n-d_2)$ many. Therefore the second claim holds.
\end{proof}

Let $\mathcal G$ be the collection of the biadjacency matrices of all $(n,m,d_1,d_2)$-bipartite biregular graphs. For fixed $u_1\in [n],v_1\in [m]$, let $\mathcal G_{u_1v_1}$ be the subset of $\mathcal G$ such that $X_{u_1v_1}=1$. We construct an edge-weighted bipartite graph $\mathfrak G_0$ on two vertex class $\mathcal G$ and $\mathcal G_{u_1v_1}$ as follows:
\begin{itemize}
    \item  If $X\in \mathcal G$ with $X_{u_1v_1}=0$,  form an edge of weight $1$ between $X$ and every element of $\mathcal G_{u_1v_1}$ that is a result of a valid forward switching from $X$.
    \item   If $X\in \mathcal G$ with $X_{u_1v_1}=1$,   form an edge of weight $d_1^2d_2(n-d_2)$ between $X$ and its identical copy in $\mathcal G_{u_1v_1}$.
\end{itemize}

Define an edge-weighted bipartite graph with two disjoint vertex sets $V_1,V_2$ to be a \textit{$(w_1,w_2)$-biregular graph} if the  degree (sum of weights from adjacent edges) of each vertex in $V_1$ is $w_1$ and the degree of each  vertex in $V_2$ is $w_2$. The following lemma shows we can embed (allowing edge weight to increase) $\mathfrak G_0$ into a weighted bipartite biregular graph.

\begin{lemma}\label{lem:G_0coupling1}
In $\mathfrak G_0$, the following holds:
\begin{enumerate}
    \item Every vertex in $\mathcal G$ has degree between $d_1^2d_2(n-2d_2)$ and  $d_1^2d_2(n-d_2)$. 
    \item Every vertex in $\mathcal G_{u_1v_1}$ has degree between $d_1^2(n-d_2)^2$ and $d_1^2n(n-d_2)$.
    \item  $\mathfrak G_0$ can be embedded into a weighted bipartite biregular graph $\mathfrak G$ on the same vertex sets, with vertices in $\mathcal G$ having degree $d_1^2d_2(n-d_2)$ and vertices in $\mathcal G_{u_1v_1}$ having degree $d_1^2n(n-d_2)$. 
\end{enumerate}  
\end{lemma}

\begin{proof} Claim
(1) follows from our construction of $\mathfrak G_0$ and the first claim in Lemma \ref{lem:switchingcounts}.  Every  $X$ in $\mathcal G_{u_1,v_1}$ with the corresponding graph $G$ is connected to $t_{u_1v_1}(G)$ many vertices in $\mathcal G$, with each edge of weight $1$. It is also connected to its identical copy in $\mathcal G$ with edge weight $d_1^2d_2(n-d_2)$. Then Claim (2) follows from the second  claim in Lemma \ref{lem:switchingcounts}. To construct $\mathfrak G$, we start with $\mathfrak G_0$ and add edges as follows. Go through the vertices of $\mathcal G$ and for each vertex with degree less than $d_1^2d_2(n-d_2)$,  arbitrarily make edges or increase edge weights  from the vertex to vertices in $\mathcal G_{u_1v_1}$ with degree less than $d_1^2n(n-d_2)$. Continue this procedure until either all vertices in $\mathcal G$ have degree $d_1^2d_2(n-d_2)$ or all vertices in $\mathcal G_{u_1,v_1}$ have degree $d_1^2n(n-d_2)$.  We claim now $\mathfrak G$ is bipartite biregular. 
From the distribution of uniform random bipartite biregular graphs, the probability of a bipartite biregular graph containing any edge $u_1v_1$ is $\frac{d_2}{n}$, we have 
$
   \frac{ |\mathcal G_{u_1v_1}|}{|\mathcal G|}=\frac{d_2}{n}.
$
If all degrees in $\mathcal G$ are $d_1^2d_2(n-d_2)$, and all degrees in $\mathcal G_{u_1,u_2,v_1}$ are at most $d_1^2n(n-d_2)$, then
$
    |\mathcal G|d_1^2d_2(n-d_2)\leq  |\mathcal G_{u_1v_1}|d_1^2n(n-d_2)=|\mathcal G|d_1^2d_2(n-d_2).
$
So all vertices in $\mathcal G_{u_1v_1}$  must have degree exactly $d_1^2n(n-d_2)$. In the same way, if all the degrees in  $\mathcal G_{u_1v_1}$  are $d_1^2n(n-d_2)$, then the degrees in $\mathcal G$ are exactly $d_1^2d_2(n-d_2)$. Therefore we can embed $\mathfrak G_0$ into a weighted  biregular bipartite graph $\mathfrak G$.
\end{proof}

In the graph $\mathfrak G$,  we uniformly choose a  random biadjacency matrix $X$ in $\mathcal G$ and consider $X^{(u_1v_1)}$ to be the element in $\mathcal G_{u_1v_1}$ given by walking from $X$ along an edge with probability proportionate to its weight. Since $\mathfrak G$ is bipartite biregular, $X^{(u_1v_1)}$ is uniformly distributed in the vertex set $\mathcal G_{u_1v_1}$. 
Lemma \ref{lem:G_0coupling1} yields a coupling of $X$ and $X^{(u_1v_1)}$ that satisfies
\begin{align}
    &\mathbb P\left(X,X^{(u_1v_1)} \text{ are identical or differ by a switching} \mid X^{(u_1v_1)}\right)\geq 1-\frac{d_1}{m},\label{eq:lowerboundp}\\
    & \mathbb P\left(X,X^{(u_1v_1)} \text{ are identical or differ by a switching} \mid X\right)\geq  1-\frac{d_1}{m-d_1}.\label{eq:lowerboundpp}
\end{align}

To see \eqref{eq:lowerboundp} holds, note that $\mathfrak G_0$ is embedded into $\mathfrak G$, from the definition of $\mathfrak G_0$, if the walk from $X$ to $X^{(u_1v_1)}$ is chosen from an edge in $\mathfrak G_0$, then $X, X^{(u_1v_1)}$ are identical or differ by a switching. Therefore the left hand side of \eqref{eq:lowerboundp} is lower bounded by the probability that an edge adjacent to $X^{(u_1v_1)}$ is chosen from $\mathcal G_0$. From (2) and (3) in Lemma \ref{lem:G_0coupling1}, such probability is at least $\frac{d_1^2(n-d_2)^2}{d_1^2n(n-d_2)}=1-\frac{d_2}{n}=1-\frac{d_1}{m}$. Similarly, \eqref{eq:lowerboundpp} holds.

\section{Concentration for linear functions of $X$}\label{sec:concentration}

Let $Q$ be a $n\times m$ matrix and $X$ be the biadjacency matrix of a uniform random $(n,m,d_1,d_2)$-bipartite biregular graph. Define a linear function for entries of $X$ as
$
    f_{Q}(X):=\sum_{u\in [n],v\in [m]}Q_{uv}X_{uv}.
$
In this section, we will use the coupling we constructed in Section \ref{sec:switching} together with Lemma \ref{lem:biasindicator}  to construct a size biased coupling of the linear function $f_Q(X)$. 

From the distribution of $X$, we have $\mathbb EX_{uv}=\frac{d_1}{m}$ for any $u\in [n],v\in [m]$. We define the following two parameters:
\begin{align}
   & \mu:=\mathbb Ef_Q(X)=\frac{d_1}{m} \sum_{u\in [n],v\in [m]}Q_{uv},\quad \tilde{\sigma}^2:=\mathbb Ef_{Q\circ Q}(X)=\frac{d_1}{m}\sum_{u\in [n],v\in [m]}Q_{uv}^2.\label{eq:sigma}
\end{align}

\begin{theorem} \label{thm:fQX1}
 Let $X$ be the biadjacency matrix of a uniform random $(n,m,d_1,d_2)$-bipartite biregular graph. Let $Q$ be a  $n\times m$  matrix with all entries in $[0,a]$. Let  $p=1-\frac{d_1}{m}$ and $p'=1-\frac{d_1}{m-d_1}$. Then for all $t\geq 0$,
 \begin{align}\label{eq:uppertail}
     \mathbb P\left( f_Q(X)-\frac{\mu}{p}\geq t \right) &\leq \exp\left(-\frac{\tilde{\sigma}^2}{3pa^2}h\left( \frac{pat}{\tilde{\sigma}^2}\right)\right),\\
     \mathbb P\left( f_Q(X)-p'\mu \leq - t \right) &\leq \exp\left(-\frac{\tilde{\sigma}^2}{3a^2}h\left( \frac{at}{\tilde{\sigma}^2}\right)\right).  \label{eq:lowertail}
 \end{align}
\end{theorem}
\begin{proof}
We construct a size biased coupling based on the analysis of switchings in Section \ref{sec:switching}. Choose a vertex $X\in \mathcal G$ uniformly at random and walk through an edge adjacent to the vertex with probability proportional to its weight. We then obtain a uniform random element $X^{(u_1v_1)}$ in $\mathcal G_{u_1v_1}$. 
The matrix $X^{(u_1v_1)}$ is distributed as $X$ conditioned on the event $X_{u_1v_1}=1.$ Independently of $X$, we choose a random variable $(U_1,V_1)$ such that for all $u\in [n],v\in [m]$,
\begin{align}\label{eq:PUV}
    \mathbb P(U_1=u,V_1=v)=\frac{Q_{uv}}{\sum_{u\in [n],v\in [m]}Q_{uv}}.
\end{align}
 Define $X'=X^{(U_1,V_1)}$. Since $X_{uv}$ is an indicator random variable for all $u\in [n],v\in [m]$, by Lemma \ref{lem:biasindicator}, the pair $(f_{Q}(X),f_{Q}(X'))$ is a size biased coupling.

   Let $\mathcal S(u_1,v_1)$ be the set of all tuples $(u_2,u_3,v_2,v_3)$ such that $(u_1,u_2,u_3,v_1,v_2,v_3)$ is a valid forward switching for $X$. Let $X(u_1,u_2,u_3,v_1,v_2,v_3)$ be the matrix obtained from  $X$ by a valid forward switching $(u_1,u_2,u_3,v_1,v_2,v_3)$. Then  for any $(u_2,u_3,v_2,v_3)\in \mathcal S(U_1,V_1)$, assuming $X_{U_1,V_1}=0$ and conditioned on $X,U_1$ and $V_1$, the element $X'\in \mathcal G_{U_1,V_1}$ is equally likely to be $X(U_1,u_2,u_3,V_1,v_2,v_3)$ for any tuple $(u_2,u_3,v_2,v_3)\in   \mathcal S(X,U_1,V_1)$. Recall from Lemma \ref{lem:G_0coupling1},  each vertex in $\mathcal  G$ has degree $d_1^2d_2(n-d_2)$ in the graph $\mathfrak G$. By our construction of the coupling  $(X,X')$, we have for any  tuple $(u_2,u_3,v_2,v_3)\in  \mathcal S (U_1,V_1)$,
  \begin{align}\label{eq:PX}
      \mathbb P(X'=X(U_1,u_2,u_3,V_1,v_2,v_3)\mid U_1,V_1,X,X_{U_1,V_1}=0)=\frac{1}{d_1^2d_2(n-d_2)}.
  \end{align}
 For any valid forward switching $(u_1,u_2,u_3,v_1,v_2,v_3)$,
 \begin{align*}
     f_Q(X(u_1,u_2,u_3,v_1,v_2,v_3))-f_Q(X) &=Q_{u_1v_1}+Q_{u_2v_3}+Q_{u_3v_2}-Q_{u_1v_2}-Q_{u_2v_1}-Q_{u_3v_3} \leq  3a.
 \end{align*}
  Let $\mathcal B$ be the event that the edge chosen in  the random walk on  $\mathfrak G$ from $X$ to $X'$ belongs to the subgraph $\mathfrak G_0$. By \eqref{eq:lowerboundp} and \eqref{eq:lowerboundpp},
 $ \mathbb P(\mathcal B\mid X')\geq p$ and $\mathbb P(\mathcal B\mid X)\geq p',$ where $p,p'$ are the parameters in the statement of Theorem \ref{thm:fQX1}.
 Therefore $f_Q(X')-f_Q(X)\leq 3a$ on the event $\mathcal B$.
  Let $ \overline{\mathcal S}(u_1,v_1)$ be the set of   $(u_2,u_3,v_2,v_3)$ such that $u_2\in\N(v_1),v_2\in \N(u_1), u_3\in \overline{\N}(v_2), v_3\in \N(u_3)$. Then  $\overline{\mathcal S}(u_1,v_1)$ has size $d_1^2d_2(n-d_2)$ and $ \mathcal S (U_1,V_1)\subset  \overline{\mathcal S}(u_1,v_1)$.
   Let $D=(f_Q(X')-f_Q(X))_+$. Then from \eqref{eq:PX},
 \begin{align*}
    & \mathbb E[D\mathbf{1}_{\mathcal B}\mid X, U_1,V_1]\\
     =&\frac{1}{d_1^2d_2(n-d_2)}\sum_{(u_2,u_3,v_2,v_3)\in   \mathcal S (U_1,V_1)}(f_{Q}(X(U_1,u_2,u_3,V_1,v_2,v_3)-f_Q(X))_+\\
     \leq  &\frac{1}{d_1^2d_2(n-d_2)}\sum_{(u_2,u_3,v_2,v_3)\in  \overline{\mathcal S}(U_1,V_1)}(Q_{U_1V_1}+Q_{u_2v_3}+Q_{u_3v_2}).
 \end{align*}
 Taking the  expectation over $U_1,V_1$, from \eqref{eq:PUV}, we have 
 \begin{align}
     &\mathbb E[D\mathbf{1}_{\mathcal B} \mid X] \notag \\
     \leq & \sum_{u_1\in [n],v_1\in [m]} \frac{Q_{u_1u_2}}{\sum_{u,v}Q_{uv}}\left(\frac{1}{d_1^2d_2(n-d_2)}\sum_{(u_2,u_3,v_2,v_3)\in \overline{\mathcal  S}(u_1,v_1)}(Q_{u_1v_1}+Q_{u_2v_3}+Q_{u_3v_2})\right) \notag\\
     =&\frac{1}{d_1^2(n-d_2)n\mu}\sum_{\substack{u_1\in [n],v_1\in [m]\\(u_2,u_3,v_2,v_3)\in  \overline{\mathcal S}(u_1,v_1)}}(Q_{u_1v_1}^2+Q_{u_1v_1}Q_{u_2v_3}+Q_{u_1v_1}Q_{u_3v_2})\label{eq:threeterms}
\end{align}
The first term in the sum \eqref{eq:threeterms} satisfies
\begin{align}\label{eq:Qu1v1}
 &\sum_{\substack{u_1\in [n],v_1\in [m]\\(u_2,u_3,v_2,v_3)\in  \overline{\mathcal S}(u_1,v_1)}}Q_{u_1v_1}^2=\sum_{u_1\in [n],v_1\in [m]}d_1^2d_2(n-d_2)Q_{u_1u_2}^2=d_1^2n(n-d_2)\tilde{\sigma}^2.
\end{align}
For the second term in \eqref{eq:threeterms},  by Cauchy's inequality,
\begin{align}\label{eq:secondsum}
 & \sum_{\substack{u_1\in [n],v_1\in [m]\\(u_2,u_3,v_2,v_3)\in  \overline{\mathcal S}(u_1,v_1)}}Q_{u_1v_1}Q_{u_2v_3}
 \leq   \left( \sum_{\substack{u_1\in [n],v_1\in [m]\\(u_2,u_3,v_2,v_3)\in  \overline{\mathcal S}(u_1,v_1)}}Q_{u_1v_1}^2\right)^{1/2}\left( \sum_{\substack{u_1\in [n],v_1\in [m]\\(u_2,u_3,v_2,v_3)\in  \overline{\mathcal S}(u_1,v_1)}} Q_{u_2v_3}^2\right)^{1/2}.
\end{align}
For a given $(u_2,v_3)$, there are $d_1^2d_2(n-d_2)$ many  $(u_1,u_3,v_1,v_2)$ such that $(u_2,u_3,v_2,v_3)\in  \overline{\mathcal S}(u_1,v_1)$. Hence 
\begin{align}
 & \sum_{\substack{u_1\in [n],v_1\in [m]\\(u_2,u_3,v_2,v_3)\in  \overline{\mathcal S}(u_1,v_1)}} Q_{u_2v_3}^2=d_1^2d_2(n-d_2)\sum_{u_2\in[n],v_3\in [m]}Q_{u_2v_3}^2=d_1^2n(n-d_2)\tilde{\sigma}^2.
\end{align}
Therefore with \eqref{eq:Qu1v1}, the left hand side of \eqref{eq:secondsum} is bounded by $d_1^2n(n-d_2)\tilde{\sigma}^2$.
By the same argument, the third term in \eqref{eq:threeterms}
\begin{align*}
    &\sum_{\substack{u_1\in [n],v_1\in [m]\\(u_2,u_3,v_2,v_3)\in  \overline{\mathcal S}(u_1,v_1)}}Q_{u_1v_1}Q_{u_3v_2}\leq d_1^2n(n-d_2)\tilde{\sigma}^2.
\end{align*}
Altogether we have \eqref{eq:threeterms} satisfies
$
    \mathbb E[D\mathbf{1}_{\mathcal B} \mid X]\leq \frac{3}{d_1^2(n-d_2)n\mu}\cdot  d_1^2n(n-d_2)\tilde{\sigma}^2=\frac{3\tilde{\sigma}^2}{\mu}.
$
\eqref{eq:uppertail} and \eqref{eq:lowertail} then follow from Theorem \ref{thm:tailestimate1} by taking $\tau^2=3\tilde{\sigma}^2, c=3a$.
\end{proof}

 \begin{corollary} \label{cor:concentration}
Let $X$ be the biadjacency matrix of a uniform random $(n,m,d_1,d_2)$-bipartite biregular graph. Let $Q$ be a real $n\times m$  matrix with all entries in $[0,a]$. Let $c_0=\frac{1}{3}(1-\frac{d_1}{m}),\gamma_0=\frac{d_1}{m-d_1}$. Then for all $t\geq 0$, we have 
\begin{align}\label{eq:oneside}
    \mathbb P(f_Q(X)-\mu&\geq \gamma_0 \mu+t)\leq \exp\left(-c_0\frac{\tilde{\sigma}^2}{a^2}h\left( \frac{at}{\tilde{\sigma}^2}\right)\right) ,\\
\label{eq:twosided}
    \mathbb P(|f_Q(X)-\mu|&\geq \gamma_0 \mu+t)\leq 2\exp\left( -\frac{c_0t^2}{2(\tilde{\sigma}^2+at/3)}\right).
\end{align}
\end{corollary}
\begin{proof}
 Recall $p=1-\frac{d_1}{m}$, $p'=1-\frac{d_1}{m-d_1}$ from Theorem \ref{thm:fQX1}. We have $c_0=\frac{p}{3}$ and $\gamma_0=\frac{1}{p}-1.$ It is shown in the proof of Proposition 2.3 (c) in \cite{cook2018size} that for any $p\in [0,1]$ and $x\geq 0$, $
    p^{-1}h(px)\geq ph(x).$
Then from \eqref{eq:uppertail}, for all $t\geq 0$, 
\begin{align}
    \mathbb P (f_Q(X)-\mu\geq \gamma_0\mu+t)&\leq \exp\left(-\frac{\tilde{\sigma}^2}{3pa^2}h\left( \frac{pat}{\tilde{\sigma}^2}\right)\right) \leq \exp\left(-c_0\frac{\tilde{\sigma}^2}{a^2}h\left( \frac{at}{\tilde{\sigma}^2}\right)\right).\label{eq:bound1}
\end{align}
Therefore \eqref{eq:oneside} holds. Note that $\gamma_0-1+p'=\frac{1}{1-d_1/m}-1-\frac{d_1}{m-d_1}= 0.$
Then from \eqref{eq:lowertail},
\begin{align}
    \mathbb P(f_Q(X)-\mu\leq -\gamma_0\mu-t) 
    &\leq \exp\left(-\frac{\tilde{\sigma}^2}{3a^2}h\left( \frac{at}{\tilde{\sigma}^2}\right)\right)\leq \exp\left(-c_0\frac{\tilde{\sigma}^2}{a^2}h\left( \frac{at}{\tilde{\sigma}^2}\right)\right). \label{eq:bound2}
\end{align}
From \eqref{eq:bound1} and \eqref{eq:bound2},  with the inequality $h(x)\geq \frac{x^2}{2(1+x/3)}$ for $x\geq 0$, we obtain
\eqref{eq:twosided}. \end{proof}

\section{The Kahn-Szemer\'{e}di argument}\label{sec:KSargument}

The Kahn-Szemer\'{e}di argument was first introduced in \cite{friedman1989second} to prove the second eigenvalue of a random $d$-regular graph is $O(\sqrt{d})$ with high probability. Later on, it has been applied to a wide range of random graph models to provide the upper bound on top eigenvalues (see for example \cite{friedman1995second,broder1998optimal,feige2005spectral,coja2009spectral,keshavan2010matrix,lubetzky2011spectra, dumitriu2013functional,lei2015consistency,cook2018size,tikhomirov2019spectral,Hoffman2019,zhou2019sparse}). We will use the concentration inequalities from Section \ref{sec:concentration} and the Kahn-Szemer\'{e}di argument \cite{friedman1989second} to prove the upper bound on $\sigma_2(X)$, which directly implies the bound for $\lambda(A)$ in  Theorem \ref{thm:main1}. 

For fixed $x\in S^{n-1}, y\in S_0^{m-1},$ we have
\begin{align}\label{eq:xXy}
  \langle x,Xy\rangle=\sum_{u\in [n],v\in [m]}X_{uv}x_uy_v.  
\end{align}

Define the set of \textit{light and heavy couples} as
\begin{align*}
    \mathcal L(x,y)&=\{(u,v): |x_uy_v|\leq \sqrt{d_1}/m \}, \quad 
     \mathcal H(x,y)=\{(u,v): |x_uy_v|> \sqrt{d_1}/m \}.
\end{align*}
 
By taking $Q=xy^{\top}$ we can decompose the linear form $f_Q(X)$ as 
\[ f_{xy^{\top}}(A)=\langle x, Xy\rangle =f_{\mathcal L(x,y)}(X)+f_{\mathcal H(x,y)}(X),\]
where 
\begin{align}
 &f_{\mathcal L(x,y)}(X)=\sum_{(u,v)\in \mathcal L(x,y)}x_uy_vX_{uv},  \quad f_{\mathcal H(x,y)}(X)=\sum_{(u,v)\in \mathcal H(x,y)}x_uy_vX_{uv}.  
\end{align}

To apply the concentration inequality, we first estimate the mean of the light part. 
\begin{lemma}\label{lem:E}
 For any fixed $x\in S^{n-1}, y\in S_0^{m-1},$ we have $ |\mathbb Ef_{\mathcal L(x,y)}(X)|\leq \sqrt{d_1}.$
\end{lemma}
\begin{proof}
Since $\mathbb EX_{uv}=\frac{d_1}{m}$ for any $u\in [n],v\in [m]$,
\begin{align*}
    |\mathbb Ef_{\mathcal L(x,y)}(X)|&=\frac{d_1}{m}\left|\sum_{(u,v)\in\mathcal L(x,y)}x_uy_v\right|
    \leq \frac{d_1}{m}\sum_{u\in [n]}\left| \sum_{v:(u,v)\in\mathcal L(x,y)} x_uy_v\right|.
\end{align*}
For any $y\in S_0^{m-1}$, we have  $\sum_{v\in [m]}y_v=0$ and it implies for fixed $u\in [n]$,
\[ \sum_{v:(u,v)\in\mathcal L(x,y)} x_uy_v=-\sum_{v:(u,v)\in\mathcal H(x,y)} x_uy_v.\]
Then
$
    |\mathbb Ef_{\mathcal L(x,y)}(X)| \leq \frac{d_1}{m}\sum_{(u,v)\in\mathcal H(x,y)} \frac{|x_uy_v|^2}{\sqrt{d_1}/m}\leq \sqrt{d_1}.
$
\end{proof}

We further split $\mathcal L(x,y)$  into two parts as  $\mathcal L(x,y)=\mathcal L_+(x,y)\cup \mathcal L_-(x,y)$ where  
\begin{align*}
\mathcal L_+(x,y)&=\{u\in [n],v\in [m]: 0\leq x_uy_v\leq \sqrt{d_1}/m \}, \quad \mathcal L_-(x,y)=\mathcal L\setminus \mathcal L_+(x,y),\\
   f_{\mathcal L_+(x,y)}(X)&=\sum_{(u,v)\in \mathcal L_+(x,y)}x_uy_vX_{uv}, \quad  f_{\mathcal L_-(x,y)}(X)=\sum_{(u,v)\in \mathcal L_-(x,y)}x_uy_vX_{uv}.
\end{align*}
 
\begin{lemma}\label{lem:52}
Let $c_0=\frac{1}{3}(1-\frac{d_1}{m}),\gamma_0=\frac{d_1}{m-d_1}$. For any fixed $(x,y)\in S^{n-1}\times S_0^{m-1}$, and $\beta \geq 2\gamma_0\sqrt{d_2}$,
\begin{align}\label{eq:lightinequality1}
    \mathbb P\left(|f_{\mathcal L(x,y)}(X)|\geq (\beta+1)\sqrt{d_1}\right)\leq 4 \exp \left( -\frac{3c_0\beta^2m}{24+4\beta}\right).
\end{align}
\end{lemma}
\begin{proof}

We have by Cauchy's inequality,
\begin{align}\label{eq:mubound}
    \mu:= &\mathbb Ef_{\mathcal L_+(x,y)}(X)\leq \frac{d_1}{m}\sum_{u\in[n],v\in [m]}|x_uy_v|\leq  d_1\sqrt{n/m}.
    \end{align}
    And 
   $
  \tilde{\sigma}^2:= \sum_{(u,v)\in \mathcal L_+(x,y)}|x_uy_v|^2\mathbb EX_{uv}\leq \frac{d_1}{m}\sum_{u\in [n],v\in [m]}|x_uy_v|^2=\frac{d_1}{m}. 
    $
For any $\beta\geq 2\gamma_0\sqrt{d_2},$  from \eqref{eq:mubound} we have $(\beta/2)\sqrt{d_1}\geq \gamma_0d_1\sqrt{n/m}\geq \gamma_0\mu.$
Then by \eqref{eq:twosided},  
\begin{align}
    &\mathbb P\left( | f_{\mathcal L_+(x,y)}(X)-\mathbb Ef_{\mathcal L_+(x,y)}(X)|\geq (\beta/2)\sqrt{d_1}\right) \notag\\
    =& \mathbb P\left( | f_{\mathcal L_+(x,y)}(X)-\mathbb Ef_{\mathcal L_+(x,y)}(X)|\geq \gamma_0\mu+ (\beta/2)\sqrt{d_1}-\gamma_0\mu\right) \notag\\
    \leq &  2\exp\left( -\frac{c_0(\frac{\beta}{2}\sqrt{d_1}-\gamma_0\mu )^2}{\frac{2d_1}{m}+\frac{2}{3}\frac{\sqrt{d_1}}{m}(\frac{\beta}{2}\sqrt{d_1}-\gamma_0\mu)}\right) 
\leq 2\exp \left( -\frac{3c_0\beta^2m}{24+4\beta}\right),
\end{align}
where in the last inequality we use the fact that $t\mapsto \frac{t^2}{a+bt}$ with $a,b>0$ is increasing for $t\geq 0$.
The same inequality holds for $\mathcal L_{-}(x,y)$. Therefore \eqref{eq:lightinequality1} holds. 
\end{proof}
We now consider the heavy part. We will follow the notations used in Section 6 of \cite{cook2018size} and begin with the definition of the discrepancy property for $X$.
\begin{definition}[discrepancy property for $X$]
Let $X$ be the biadjacency matrix of a $(n,m,d_1,d_2)$-bipartite biregular graph. For any $S\subset [n]$,  $T\subset [m],$ define
$
    e(S,T):=\sum_{u\in S,v\in T}X_{uv}.
$
We say $A$ satisfies the \textit{discrepancy property} with  $\delta\in(0,1),\kappa_1<1$ and $\kappa_2\geq 0$, denoted by $\textnormal{DP}(\delta,\kappa_1,\kappa_2)$, if for all non-empty subsets $S\subset [n], T\subset [m]$, at least one of the following holds:
\begin{enumerate}
    \item $e(S,T)\leq \kappa_1\delta |S||T|$,\label{case1}
    \item $e(S,T)\log\left(\frac{e(S,T)}{\delta |S||T|}\right)\leq \kappa_2(|S|\vee |T|) \log \left(\frac{em}{|S|\vee |T|}\right)$.\label{case2}
\end{enumerate}
\end{definition}
Since $X$ is not Hermitian, following the same proof from Lemma 6.4 in \cite{cook2018size} with some modification for our model, we obtain the following lemma.
\begin{lemma}\label{lem:DPproperty}
Let $X$ be the biadjacency matrix of a random $(n,m,d_1,d_2)$-bipartite biregular graph. For any $K> 0$,
with probability at least $1-m^{-K}$, the discrepancy property $\textnormal{DP}(\delta,\kappa_1,\kappa_2)$ holds for $X$ with $\delta=\frac{d_1}{m}$,
$\kappa_1=e^2(1+\gamma_0)^2,$ and $ \kappa_2=\frac{2}{c_0}(1+\gamma_0)(K+4)$.
\end{lemma}
\begin{proof} 
Let $Q=\mathbf{1}_S\mathbf{1}_{T}^{\top},$ where $\mathbf{1}_S\in \{0,1\}^n$ and $\mathbf{1}_T\in \{0,1\}^m$ are  the indicator vectors of the set $S$ and $T$, respectively.  We have 
\begin{align*}
    f_Q(X)=\sum_{u\in [n],v\in [m]}(\mathbf{1}_S)_{u}(\mathbf{1}_T)_{v}X_{uv}=e(S,T).
\end{align*}
Denote 
$\mu(S,T):=\mathbb Ee(S,T)=\frac{d_1}{m}|S||T|=\delta |S||T|.$
Recall $\gamma_0=\frac{d_1}{m-d_1}$.
For fixed $K>0$, let $\gamma_1=e^2(1+\gamma_0)^2-1$ and $\gamma=\gamma(S,T,m)=\max(\gamma^*,\gamma_1)$, where $\gamma^*$ is the unique $x$ such that
\begin{align}\label{eq:identity}
    c_0h(x-\gamma_0)\mu(S,T)=(K+4)(|S|\vee |T|)\log \left(\frac{em}{|S|\vee |T|}\right).
\end{align}
Note that we have  $\gamma\geq \gamma_1\geq \gamma_0$. Taking $a=1$ and 
$\tilde{\sigma}^2=\frac{d_1}{m}\sum_{u\in [n],v\in [m]}Q_{uv}^2=\frac{d_1}{m}|S||T|=\mu(S,T)$
in \eqref{eq:oneside}, we obtain
\begin{align*}
    \mathbb P(e(S,T)\geq (1+\gamma)\mu(S,T))=& \mathbb P(f_Q(X)\geq (1+\gamma)\mathbb Ef_Q(X))\\
    \leq & \exp\left(-c_0\tilde{\sigma}^2h\left(
    \frac{(\gamma-\gamma_0)\mu(S,T)}{\tilde{\sigma}^2}\right)\right)
    =\exp\left(-c_0\mu(S,T)h\left(
    \gamma-\gamma_0\right)\right).
\end{align*}
Then 
\begin{align*}
  &\mathbb P(\exists S\subset [n], T\subset [m]  ,|S|=s, |T|=t, e(S,T)\geq (1+\gamma)\mu(S,T))\\
  \leq  &\sum_{S\subset [n],|S|=s}\sum_{T\subset [m], |T|=t}\exp(-c_0h(\gamma-\gamma_0)\mu(S,T))   \\
  \leq &\sum_{S\subset [n],|S|=s}\sum_{T\subset [m], |T|=t}\exp(-c_0h(\gamma^*-\gamma_0)\mu(S,T)) \\
  = & {n\choose s}{m\choose t}\exp\left(-(K+4)(s\vee t)\log \left(\frac{em}{s\vee t}\right)\right)
  \leq \left(\frac{ne}{s}\right)^s\left(\frac{me}{t}\right)^t\exp\left(-(K+4)(s\vee t)\log \left(\frac{em}{s\vee t}\right)\right)\\
     \leq & \exp\left(-(K+2)(s\vee t)\log \left(\frac{em}{s\vee t}\right)\right)\leq \exp(-(K+2)\log(em)),
\end{align*}
where in the third line we use the fact that $h(x)$ is increasing for $x\geq 0$, and in the last line we 
use the fact that $x\mapsto x\log(e/x)$ is increasing on $[0,1]$. Taking a union bound over all $s\in [n],t\in[m]$, we have 
\begin{align}\label{eq:event}
 &\mathbb P(\exists S\in [n], T\in [m],  e(S,T)\geq (1+\gamma)\mu(S,T)) \leq  nm\exp(-(K+2)\log (em))\leq m^{-K}.  
\end{align}
If the subsets $S,T$ satisfy $\gamma(S,T,m)=\gamma_1$, then from \eqref{eq:event} with probability at least $1-m^{-K}$, 
\begin{align*}
    e(S,T)\leq (1+\gamma_1)\mu(S,T)= e^2(1+\gamma_0)^2\delta |S||T|.
\end{align*}
Hence Case \ref{case1} holds with $\kappa_1=e^2(1+\gamma_0)^2$. Now assume the subsets $S,T$ satisfy $\gamma(S,T,m)=\gamma^*.$ From \eqref{eq:event},  with probability at least $1-m^{-K}$,
$e(S,T)\leq (1+\gamma^*)\mu(S,T).$ Therefore from  \eqref{eq:identity},
\begin{align}\label{eq:h}
    \frac{c_0}{1+\gamma^*} h(\gamma^*-\gamma_0)e(S,T)\leq (K+4)(|S|\vee |T|)\log \left(\frac{em}{|S|\vee |T|}\right).
\end{align}
It is shown in the proof of Lemma 6.4 in \cite{cook2018size} that 
\[ \frac{h(\gamma^*-\gamma_0)}{1+\gamma_*}\geq \frac{1}{2(1+\gamma_0)} \log\frac{e(S,T)}{\mu(S,T)}=\frac{1}{2(1+\gamma_0)} \log\frac{e(S,T)}{\delta|S||T|}.\]
Together with \eqref{eq:h}, it implies that when $\gamma^*\geq \gamma_1,$
\begin{align*}
    e(S,T)\log \frac{e(S,T)}{\delta|S||T|)}\leq \frac{2}{c_0}(1+\gamma_0)(K+4)(|S|\vee |T|)\log \frac{em}{|S|\vee|T|}.
\end{align*}
Then Case \ref{case2} follows with $\kappa_2=\frac{2}{c_0}(1+\gamma_0)(K+4)$. This completes the proof.
\end{proof}

Assuming the discrepancy property holds for $X$, the following lemma implies that the contribution from heavy tuples is $O(\sqrt{d_1})$. Since the proof is very similar to the proof of Lemma 6.6 in \cite{cook2018size}, we omit the details.

\begin{lemma}\label{lem:DPa}
Let $X$ be the biadjacency matrix of a $(n,m,d_1,d_2)$-bipartite biregular graph. Suppose $X$ has the discrepancy property $\textnormal{DP}(\delta,\kappa_1,\kappa_2)$ with $\delta, \kappa_1,\kappa_2$ given in Lemma \ref{lem:DPproperty}. Then for any $(x,y)\in S^{n-1}\times S_0^{m-1}$,  we have
$f_{\mathcal H(x,y)}(X)\leq \alpha_0\sqrt{d_1}$
with $\alpha_0=48+32\kappa_1+64\kappa_2\left(1+\frac{1}{\kappa_1\log\kappa_1} \right).$
\end{lemma}

Now we finish the proof of Theorem \ref{thm:main1} with the $\epsilon$-net argument. The following lemma is standard and we omit the proof. 
\begin{lemma}\label{eq:epsilonnetapprox}
For $\epsilon\in (0,1/2)$, let  $\N_{\epsilon}$ be an $\epsilon$-net of $S^{n-1}$ and $\N_{\epsilon}^0$ be an $\epsilon$-net of $S_0^{m-1}$. Let $X$ be the biadjacency matrix of a $(n,m,d_1,d_2)$-bipartite biregular graph. Then
\begin{align}
   \sigma_2(X)=\sup_{x\in S^{n-1},y\in S_0^{m-1}} \langle x,Xy\rangle\leq \frac{1}{1-2\epsilon}\sup_{x\in \N_{\epsilon}, y\in \N_{\epsilon}^0} |\langle x, Xy\rangle |. \notag
\end{align}
\end{lemma}

\begin{proof}[Proof of Theorem \ref{thm:main1}]
 Fix $K>0$. By Lemma \ref{lem:DPproperty}, with probability at least $1-m^{-K}$, $X$ has $\textnormal{DP}(\delta,\kappa_1,\kappa_2)$ property. Let $\mathcal D$ be the event that this property holds. Take $\epsilon=1/4$ in Lemma \ref{eq:epsilonnetapprox}.
 By taking the union bound over $\N_{\epsilon}\times \N_{\epsilon}^0$, we have for any $\alpha>0$,
 \begin{align} 
     \mathbb P\left(\mathcal D\cap \left\{ \sigma_2(X)\geq \alpha\sqrt{d_1}\right\}\right)&\leq \sum_{x\in \N_{\epsilon}, y\in \N_{\epsilon}^0}\mathbb P\left(\mathcal D\cap \left\{ |\langle x,Xy\rangle |\geq \alpha/2\sqrt{d_1}\right\}\right).
 \end{align}
 For any fixed $(x,y)\in S^{n-1}\times S_0^{m-1}$ and any $\alpha>2\alpha_0$ with $\alpha_0$ given in Lemma \ref{lem:DPa}, from our analysis of the heavy couples, we have 
 \begin{align*}
 \mathbb P\left(\mathcal D\cap \left\{ |\langle x,Xy\rangle |\geq \alpha/2\sqrt{d_1}\right\}\right)&\leq  \mathbb P\left(\mathcal D\cap \left\{ |f_{\mathcal L(x,y)}(X)|\geq \alpha/2\sqrt{d_1}-|f_{\mathcal H(x,y)}(X)|\right\}\right)\\
 &\leq \mathbb P\left( |f_{\mathcal L(x,y)}(X)|\geq (\alpha/2-\alpha_0)\sqrt{d_1} \right).
 \end{align*}
 Take $\beta=\alpha/2-\alpha_0-1$. When $\beta>2\gamma_0\sqrt{d_2}$,  from \eqref{eq:lightinequality1}, we have 
 \begin{align*}
\mathbb P\left( |f_{\mathcal L(x,y)}(X)|\geq (\alpha/2-\alpha_0)\sqrt{d_1}\right)\leq 4\exp\left(-\frac{3c_0\beta^2m}{24+4\beta}\right).
 \end{align*}
When $\epsilon=1/4$, there exist $\epsilon$-nets $\N_{\epsilon}$ and $\N_{\epsilon}^0$ such that $|\mathcal N_{\epsilon}|\leq 9^n$ and  $|\mathcal N_{\epsilon}^0|\leq 9^m$.  Then 
 \begin{align}\label{eq:sigmabeta}
     \mathbb P\left(\mathcal D\cap \left\{ \sigma_2(A)\geq \alpha\sqrt{d_1}\right\}\right)&\leq 4\cdot 9^{n+m}\exp\left(-\frac{3c_0\beta^2m}{24+4\beta}\right)\leq  9^{2m+1}\exp\left(-\frac{3c_0\beta^2m}{24+4\beta}\right).
 \end{align}
  Then
$c_0\geq \frac{1}{6}, 2\gamma_0 \sqrt{d_2}= \frac{2d_2\sqrt{d_2}}{n-d_2}\leq \frac{4d_2\sqrt{d_2}}{n}\leq \sqrt{2}.$
Recall $\kappa_1=e^2(1+\gamma_0)^2, \kappa_2=\frac{2}{c_0}(1+\gamma_0)(K+4)$. We have 
$e^2\leq \kappa_1\leq 4e^2,  \kappa_2\leq 24(K+4) .$
Then $\alpha_0=32\kappa_1+48+64\kappa_2\left(1+\frac{1}{\kappa_1\log\kappa_1}\right)$
is a  constant depending on  $K$. Note that $\alpha=2\beta+2\alpha_0+2$.  Take $\beta>\sqrt{2}\geq 2\gamma_0\sqrt{d_2}$ to be a sufficiently large constant independent of $m$ such that
$9^{2m+1}\exp\left(-\frac{3c_0\beta^2m}{24+4\beta}\right)\leq e^{-m}. $
Then $\alpha$ is a constant  depending only on $K$. From \eqref{eq:sigmabeta}, 
 \begin{align*}
  \mathbb P\left( \sigma_2(X)\geq \alpha\sqrt{d_1} \right)\leq \mathbb P(\mathcal D^c)+  \mathbb P\left(\mathcal D\cap \left\{ \sigma_2(X)\geq \alpha\sqrt{d_1}\right\}\right)\leq m^{-K}+e^{-m}. 
 \end{align*}
 This completes the proof of Theorem \ref{thm:main1}.
\end{proof}

\section{Higher order switching}\label{sec:highswitching}

{The size-biased coupling can be used to study any random variable that permits a size-biased coupling with good properties, see the general theorem stated in Lemma \ref{thm:tailestimate1}, and more  applications of size-biased coupling in \cite{cook2018size,arratia2019size,ghosh2011concentration}.
From  Lemma \ref{lem:biasindicator}, it's easy to construct such sized-biased coupling for linear forms of $A$, since it is a weighted sum of indicator random variables, and the switching operators were used to create a coupling that has good control of $\mathcal B$ and $D$ in the statement of Lemma \ref{thm:tailestimate1}.
 The limitation of this argument is that one cannot get a sharp constant in front of $\sqrt{d_1-1}+\sqrt{d_2-1}$ with the use of  an $\epsilon$-net. This size-biased coupling technique is used to prove concentration of  $x^\top Ay$ for a random matrix $A$ and two fixed vectors $x,y$. After passing through the $\epsilon$-net, the constants in front of $\sqrt{d_1-1}+\sqrt{d_2+1}$ will depend on the size of the net.	
 It's possible that one could come up with a better size-biased coupling without taking a union bound over the $\epsilon$-net to get the sharp constant.
 
 In this section, we modify our switching operations and apply the previous analysis to $M:=XX^\top-d_1I$. By applying an $\epsilon$-net argument to $M$, we won't be able to capture the sharp dependence on $d_2$, but it allows us to find a sharp dependence on $d_1$.

  Now we move on to the proof of Theorem \ref{thm:main2}. 
For a given $n\times n$ symmetric matrix $Q$, the linear function $f_Q(M)$ can be written as
\begin{align*}
    f_Q(M)&=\sum_{u_1,u_2\in [n]}Q_{u_1u_2}M_{u_1u_2}=\sum_{u_1\not= u_2,v\in [m]}Q_{u_1u_2}X_{u_1v}X_{u_2v}.
\end{align*}
We see that $f_Q(M)$ is a linear combination of $n(n-1)m$ many indicator random variables $X_{u_1v}X_{u_2v}$. Then to construct a size biased coupling for $f_Q(M)$, according to Lemma \ref{lem:biasindicator}, we need to construct a coupling $(X,X^{(u_1u_2v_1)})$ where $X^{(u_1u_2v_1)}$ is distributed as $X$ conditioned on the event $X_{u_1v_1}X_{u_2v_1}=1.$  

To create a size biased coupling with the desired property, we introduce three types of switchings involving more vertices.

\begin{definition}[valid Type 1 and Type 2 switchings]\label{def:12}
Assume $X_{u_1v_1}=X_{u_2v_2}=X_{u_3v_1}=X_{u_4v_3}=1, X_{u_2v_1}=X_{u_3v_3}=X_{u_4v_2}=0.$ We define 
$(u_1,u_2,u_3,u_4,v_1,v_2,v_3)$ to be a  \textit{valid Type 1 forward switching} if after the switching, the edges $u_2v_1, u_3v_3, u_4v_2$ are added and the edges $u_2v_2,u_3v_1,u_4v_3$ are removed. 
 
For valid Type 1 switchings, the appearance of the edge $u_1v_1$ is not changed. Similarly, if 
$X_{u_1v_2}=X_{u_2v_2}=X_{u_3v_1}=X_{u_4v_3}=1, X_{u_1v_1}=X_{u_3v_3}=X_{u_4v_2}=0.$
We define 
$(u_1,u_2,u_3,u_4,v_1,v_2,v_3)$ to be a  \textit{valid Type 2 forward switching}. 
For valid Type 2 switchings, the appearance of the edge $u_2v_1$ is not changed. Similar to Definition \ref{def:switch}, the backward switchings are defined accordingly.
\end{definition}
\begin{figure}[ht]
    \centering
    \includegraphics[width=0.45\linewidth]{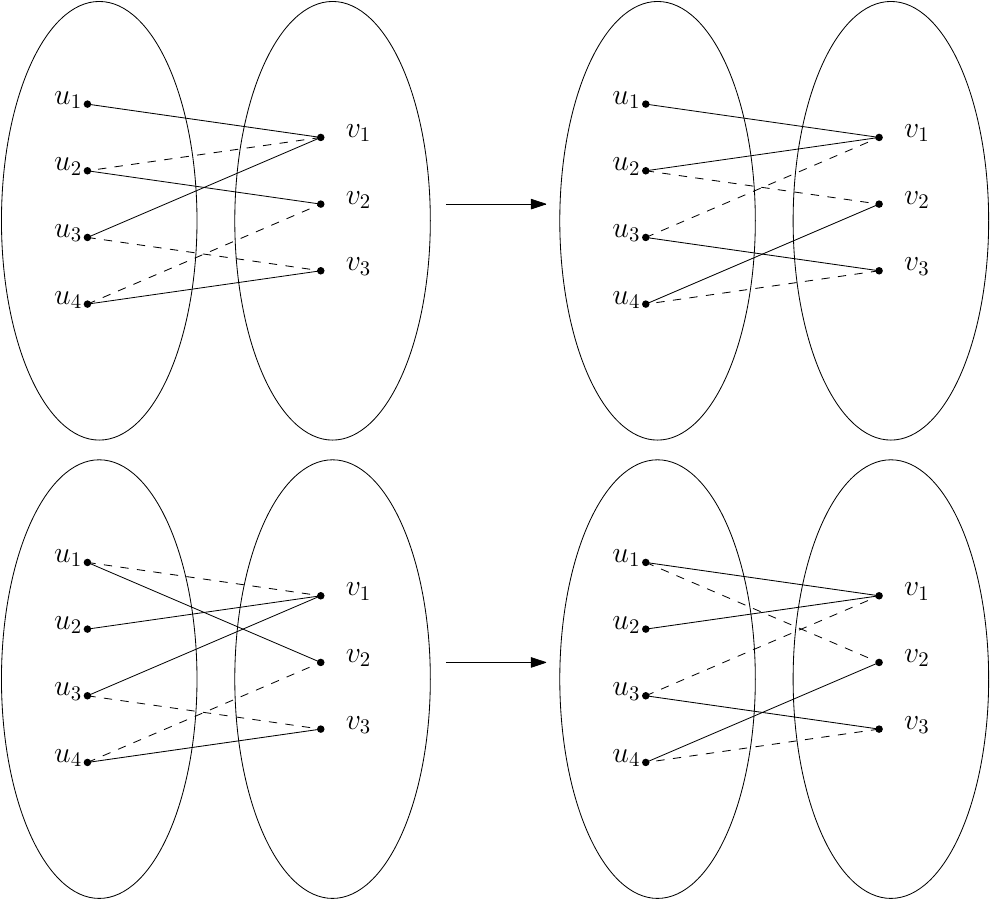}
    \caption{ valid Type 1 and Type 2 forward switchings for fixed $u_1,u_2,v_1$}
    \label{fig:my_label}
\end{figure}

See Figure \ref{fig:my_label} for an example of Definition \ref{def:12}. In addition, we  define the following  switching as a combination of Type 1 and Type 2 switchings.
\begin{definition}[valid Type 3 switchings]\label{def:type3}
Assume $X_{u_1v_4}=X_{u_2v_2}=X_{u_3v_1}=X_{u_4v_3}=X_{u_5v_1}=X_{u_6v_5}=1$ and 
$X_{u_1v_1}=X_{u_2v_1}=X_{u_3v_3}=X_{u_4v_2}=X_{u_5v_5}=X_{u_6v_4}=0.$ we define 
$(u_1,u_2,u_3,u_4,u_5,u_6,v_1,v_2,v_3,v_4,v_5)$ to be a  \textit{valid Type 3 forward switching} if after the switching, the edges $u_1v_1, u_2v_1, u_3v_3,u_4v_2,u_5v_5,u_6v_4$ are added and the edges $u_2v_1,u_3v_1,u_4v_3$ are removed.  
\end{definition}
\begin{figure}[ht]
    \centering
    \includegraphics[width=0.5\linewidth]{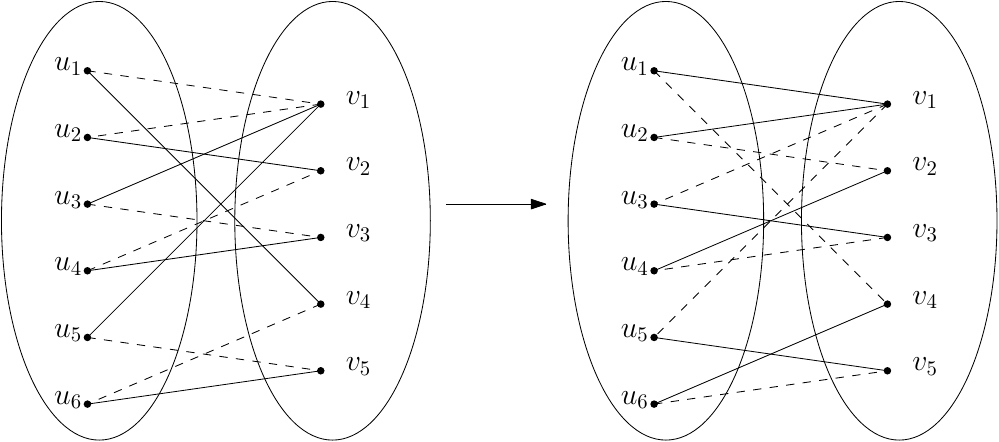}
    \caption{a valid Type 3 forward switching for fixed $u_1,u_2,v_1$}
    \label{fig:my_label12}
\end{figure}

See Figure \ref{fig:my_label12} for an example of Definition \ref{def:type3}. 
 We first estimate the number of all types of switchings for a given bipartite biregular graph. The argument is similar to the proof of Lemma \ref{lem:switchingcounts}.

\begin{lemma} \label{lem:switchingcounts1}
Consider an $(n,m,d_1,d_2)$-bipartite biregular graph $G$. Let $s^{(i)}_{u_1u_2v_1}(G), 1\leq i\leq 3$ be the number of valid Type $i$ forward switchings of the form
\[
    (u_1,u_2,\cdot,\cdot,v_1,\cdot,\cdot),\quad  (u_1,u_2,\cdot,\cdot,v_1,\cdot,\cdot), \quad (u_1,u_2,\cdot,\cdot,\cdot,\cdot,v_1,\cdot,\cdot,\cdot,\cdot),\]
respectively. Let $t^{(i)}_{u_1u_2v_1}(G)$ be the number of the corresponding valid Type $i$ backward switchings for $i=1,2,3$. Then the following inequalities hold:
\begin{enumerate}
    \item If $X_{u_1v_1}=1, X_{u_2v_1}=0$,
    \begin{align}
       d_1^2(d_2-1)(n-2d_2)\leq  s^{(1)}_{u_1u_2v_1}\leq d_1^2(d_2-1)(n-d_2).\label{eq:suv1}
    \end{align}
 If $X_{u_1v_1}=0, X_{u_2v_1}=1$,
    \begin{align}\label{eq:suv2}
d_1^2(d_2-1)(n-2d_2)\leq  s^{(2)}_{u_1u_2v_1}\leq d_1^2(d_2-1)(n-d_2).    \end{align}
      \item If $X_{u_1v_1}=0=X_{u_2v_1}=0$,
    \begin{align}
       d_1^4d_2(d_2-1)(n-d_2)(n-3d_2)\leq  s^{(3)}_{u_1u_2v_1}\leq d_1^4d_2(d_2-1)(n-d_2)^2.\label{eq:s3}
       \end{align}
         \item If $X_{u_1v_1}=X_{u_2v_1}=1$,
    \begin{align}
       d_1^2(n-d_2)(n-2d_2)\leq   &t^{(1)}_{u_1u_2v_1}\leq d_1^2(n-d_2)^2,\label{eq:t1}\\
       d_1^2(n-d_2)(n-2d_2)\leq   &t^{(2)}_{u_1u_2v_1}\leq d_1^2(n-d_2)^2,\label{eq:t2}
       \end{align}
       and 
       \begin{align}
        d_1^4(n-d_2)^2(n-d_2-1)(n-3d_2) \leq   &t^{(3)}_{u_1u_2v_1}\leq  d_1^4(n-d_2)^3(n-d_2-1) .\label{eq:t3}
    \end{align}
\end{enumerate}
 
\end{lemma}
\begin{proof} 
  For fixed $u_1,u_2,v_1$, assume $X_{u_1v_1}=1,X_{u_2v_1}=0$. We first consider $s^{(1)}_{u_1u_2v_1}$. We bound the number of all possible Type 1 forward switching by choosing $u_3\in \N(v_1)$ with $u_3\not=u_1$, $v_2\in \N(u_2)$,  $u_4\in \overline{\N}(v_2)$, and  $v_3\in \N(u_4)$. Here we require $u_3\not=u_1$  because after the switching, $u_3,v_1$ are not connected but $u_1,v_1$ are connected (see Figure \ref{fig:my_label}). There are $d_1^2(d_2-1)(n-d_2)$ many choices in total. This gives the upper bound on  $s^{(1)}_{u_1u_2v_1}$.  
    For the lower bound, we consider the number of tuples that do not allow a valid Type 1 forward switching among the $d_1^2(d_2-1)(n-d_2)$ many choices, and  we denote it by $K$. Then
$s^{(1)}_{u_1u_2v_1}=d_1^2(d_2-1)(n-d_2)-K.$
Any tuple chosen as above is not a valid switching if and only if $v_3\in \N(u_3)$. We upper bound $K$ by choosing $v_2\in\N(u_2)$, $u_3\in \N(v_1)$ with $u_3\not=u_1$, $v_3\in \N(u_3)$ and $u_4\in \N(v_3)$, making at most $d_1^2d_2(d_2-1)$ many choices in total. Hence $K\leq d_1^2d_2(d_2-1)$ and  $s^{(1)}_{u_1u_2v_1}\geq d_1^2(d_2-1)(n-2d_2).$
The estimate of Type 2 forward switchings follows from the symmetry of $u_1$ and $u_2$. Therefore \eqref{eq:suv1}  and \eqref{eq:suv2} hold.

Now assume $X_{u_1v_1}=X_{u_2v_1}=1$.  For fixed $(u_1,u_2,v_1)$, by choosing $ u_3\in \overline{\N}(v_1)$, $v_3\in \N(u_3)$,  $u_4\in \overline{\N}(v_3)$ and $v_2\in \N(u_4)$, there are at most $d_1^2(n-d_2)^2$ many choices.  
Among those tuples, the switching will fail if and only if $v_2\in \N(u_2)$. Let 
$t_{u_1u_2v_1}^{(1)}=d_1^2(n-d_2)^2-L,$
where $L$ is the number of invalid Type 1 backward switchings among the $d_1^2(n-d_2)^2$ many choices above. We can then bound $L$ by choosing $u_3\in \overline{\N}(v_1),v_2\in \N(u_2),v_3\in \N(u_3)$ and $u_4\in \N(v_2)$, which has at most $d_1^2d_2(n-d_2)$ many choices. Hence 
$
    d_1^2(n-d_2)(n-2d_2)\leq  t_{u_1u_2v_1}^{(1)}\leq d_1^2(n-d_2)^2.
$
The estimate of $t_{u_1u_2v_1}^{(2)}$ follows in the same way.  Then \eqref{eq:t1} and \eqref{eq:t2} hold.

It remains to show \eqref{eq:s3} and \eqref{eq:t3}. Assume $X_{u_1v_1}=X_{u_2v_1}=0$. By choosing $u_3\in \N(v_1), v_2\in \N(u_2), u_4\in \overline{\N}(v_2), v_3\in \N(u_4)$, $u_5\in \N(v_1)$ with $u_5\not=u_3$, $v_4\in\N(u_1), u_6\in \overline{\N}(v_4)$ and $v_5\in N(u_6)$, we have a total number of $d_1^2d_2(d_2-1)(n-d_2)^2$ tuples.  Therefore $s_{u_1u_2v_1}^{(3)}\leq d_1^4d_2(d_2-1)(n-d_2)^2.$
Among those tuples, a tuple is not valid Type 3 forward switching if and only if one of the two cases happens: (i) $v_3\in \N(u_3)$ or (ii) $v_5\in \N(u_5)$.  
Let 
$s_{u_1u_2v_1}^{(3)}= d_1^4d_2(d_2-1)(n-d_2)^2-K.$ For the first case, by choosing $u_3\in \N(v_1), v_2\in \N(u_2),v_3\in \N(u_3),u_4\in \N(v_3), u_5\in \N(v_1)$ with $u_5\not=u_3$, $v_4\in \N(v_1)$, $u_6\in \overline{\N}(v_4)$ and $v_5\in \N(u_5)$  there are at most $d_1^4d_2^2(d_2-1)(n-d_2)$ many tuples that are not valid. 
For the second case,  by choosing $u_3\in \N(v_1), v_2\in \N(u_2),u_4\in \overline{\N}(v_2),v_3\in \N(u_4), u_5\in \N(v_1)$ with $u_5\not=u_3$, $v_4\in \N(u_1)$, $v_5\in \N(u_5)$ and $u_6\in \N(v_5)$  there are at most $d_1^4d_2^2(d_2-1)(n-d_2)$ many tuples that are not valid.  From those two cases, we have
$K\leq 2d_1^4d_2^2(d_2-1)(n-d_2)$ and 
$ s_{u_1u_2v_1}^{(3)}\geq d_1^4d_2(d_2-1)(n-d_2)(n-3d_2).$
This implies \eqref{eq:s3}.

Now assume $X_{u_1v_1}=X_{u_2v_1}=1$. For the upper bound on  $t_{u_1u_2v_1}^{(3)}$,  we choose $u_3\in \overline{\N}(v_1), v_3\in \N(u_3), u_4\in \overline{\N}(v_3), v_2\in \N(u_4), u_5\in \overline{\N}(v_1)$ with $u_5\not=u_3$, $v_5\in \N(u_5), u_6\in \overline{\N}(v_5)$ and $v_4\in \N(u_6)$. This gives a total number of $d_1^4(n-d_2)^3(n-d_2-1)$ choices. Among those choices, a tuple is  not a valid Type 3 backward switching if and only if (i) $v_2\in \N(u_2)$ or (ii) $u_4\in \N(u_1).$  
Let 
$ t_{u_1u_2v_1}^{(3)}=d_1^4(n-d_2)^3(n-d_2-1)-L.$
For the first case, we have at most $d_1^4d_2(n-d_2)^2(n-d_2-1)$ choices by choosing $u_3\in \overline{\N}(v_1), v_2\in \N(u_2),v_3\in \N(u_3), u_4\in \N(v_2),u_5\in \overline{\N}(v_1),u_5\not=u_3, v_5\in \N(u_5),u_6\in \overline{\N}(v_5)$ and $v_4\in \N(u_6)$. By a similar argument, there are at most $d_1^4d_2(n-d_2)^2(n-d_2-1)$ choices for the second case. Hence $L\leq 2d_1^4d_2(n-d_2)^2(n-d_2-1)$ and 
$t_{u_1u_2v_1}^{(3)}\geq d_1^4(n-d_2)^2(n-d_2-1)(n-3d_2).$
Therefore \eqref{eq:t3} holds.
\end{proof}

Let $\mathcal G$ be the collection of the biadjacency matrices  of all $(n,m,d_1,d_2)$-bipartite biregular graphs and $\mathcal G_{u_1u_2v_1}$ be the subset of $\mathcal G$ such that $X_{u_1v_1}=X_{u_2v_1}=1$. Fix $u_1,u_2\in [n]$ with $u_1\not=u_2$ and $v_1\in [m]$. We construct an edge-weighted bipartite graph $\mathfrak G_0$ on two vertex class $\mathcal G$ and $\mathcal G_{u_1u_2v_1}$ as follows:
\begin{itemize}
    \item If $X\in \mathcal G$ with $X_{u_1v_1}=0, X_{u_2v_1}=1$, then form an edge of weight $d_1^2d_2(n-d_2)$ between $X$ and every element of $\mathcal G_{u_1u_2v_1}$ that is a result of a valid Type 1 forward switching from $X$.
    \item If $X\in \mathcal G$ with $X_{u_1v_1}=1, X_{u_2v_1}=0$, then form an edge of weight $d_1^2d_2(n-d_2)$ between $X$ and every element of $\mathcal G_{u_1u_2v_1}$ that is a result of a valid Type 2 forward switching from $X$.
     \item If $X\in \mathcal G$ with $X_{u_1v_1}=X_{u_2v_1}=0$, then form an edge of weight $1$ between $X$ and every element of $\mathcal G_{u_1u_2v_1}$ that is a result of a valid Type 3 forward switching from $X$.
      \item If $X\in \mathcal G$ with $X_{u_1v_1}=X_{u_2v_1}=1$, then form an edge of weight $d_1^4d_2(d_2-1)(n-d_2)^2$ between $X$ and its identical copy in $G_{u_1u_2v_1}$.
\end{itemize}
By our construction of $\mathfrak G_0$, the following lemma holds. 

\begin{lemma}\label{lem:G_0coupling}
In $\mathfrak G_0$, the following holds:
\begin{enumerate}
    \item Every element in $\mathcal G$ has degree between $d_1^4d_2(d_2-1)(n-d_2)(n-3d_2)$ and  $d_1^4d_2(d_2-1)(n-d_2)^2$.
    \item Every element in $\mathcal G_{u_1u_2v_1}$ has degree between $d_1^4(n-d_2)^2(n-2d_2)(n-1)$ and $d_1^4(n-d_2)^2n(n-1)$.
    \item  $\mathfrak G_0$ can be embedded into a weighted bipartite biregular graph $\mathfrak G$ on the same vertex sets, with vertices in $\mathcal G$ having degree $d_1^4d_2(d_2-1)(n-d_2)^2$ and vertices in $\mathcal G_{u_1u_2v_1}$ having degree $d_1^4(n-d_2)^2n(n-1)$. 
\end{enumerate}  
\end{lemma}
\begin{proof}
From Lemma \ref{lem:switchingcounts1}, with the weight we assign to each edge, we have the following:
\begin{itemize}
    \item If $X_{u_1,v_1}=0, X_{u_2,v_2}=1$ or  $X_{u_1,v_1}=1, X_{u_2,v_2}=0$, then the degree of $X$ in $\mathfrak G_0$ is between $d_1^4d_2(d_2-1)(n-d_2)(n-2d_2)$ and $d_1^4d_2(d_2-1)(n-d_2)^2$.
    \item If $X_{u_1,v_1}=X_{u_2,v_2}=0$, from \eqref{eq:s3}, the degree of $X$ in $\mathfrak G_0$ is between $d_1^4d_2(d_2-1)(n-d_2)(n-3d_2)$ and $d_1^4d_2(d_2-1)(n-d_2)^2$.
    \item If $X_{u_1,v_1}=X_{u_2,v_2}=1$, then the degree of $X$ in $\mathfrak G_0$ is exactly $d_1^4d_2(d_2-1)(n-d_2)^2$.
\end{itemize}
Therefore the first claim of Lemma \ref{lem:G_0coupling} holds.
 Now we turn to elements in $G_{u_1u_2v_1}$. For any $X\in \mathcal G_{u_1u_2v_1}$, $X$ can be adjacent to elements $X'$ such that
\begin{itemize}
    \item  $X'$ is a result of a valid Type 1 or Type 2 backward switching from $X$, with edge weight $d_1^2d_2(n-d_2)$.
    \item $X'$ is a result of a  valid Type 3 backward switching from $X$, with edge weight $1$.
    \item $X'$ is an identical copy of $X$ in $\mathcal G$, with edge weight $d_1^4d_2(d_2-1)(n-d_2)^2$.
\end{itemize}
Then for each $X\in \mathcal G_{u_1u_2v_1}$, combining all the weighted edges from valid Type 1,2,3 backward switchings and the weighted edge from its identical copy, from \eqref{eq:t1}, \eqref{eq:t2} and \eqref{eq:t3}, the degree is at most
\begin{align*}
     &2d_1^2d_2(n-d_2)\cdot d_1^2(n-d_2)^2+d_1^4(n-d_2)^3(n-d_2-1)+d_1^4d_2(d_2-1)(n-d_2)^2\\
    =& d_1^4(n-d_2)^2n(n-1).
\end{align*}
And similarly, its degree is at least $d_1^4(n-d_2)^2(n-2d_2)(n-1)$.

Then the second claim in Lemma \ref{lem:G_0coupling} holds.  It remains to prove the third claim. Since $X$ is uniformly distributed, we have
\begin{align*}
    \mathbb P(X_{u_1v_1}=X_{u_2v_1}=1)=\frac{nd_1(d_2-1)}{n(n-1)m}=\frac{d_2(d_2-1)}{n(n-1)}
\end{align*}
and
\begin{align}\label{eq:ratio}
    \frac{|\mathcal G_{u_1u_2v_1}|}{|\mathcal G|}=\frac{d_2(d_2-1)}{n(n-1)}.
\end{align}
To construct $\mathfrak G$, we start with $\mathfrak G_0$ and add edges as follows. Go through the vertices in $\mathcal G$ and for each vertex with degree less than $d_1^4d_2(d_2-1)(n-d_2)^2$, arbitrarily make edges  from the vertex to vertices in $\mathcal G_{u_1,u_2,v_1}$ with degree less than $d_1^4(n-d_2)^2n(n-1)$. Continue this procedure until either all vertices in $\mathcal G$ have degree $d_1^4d_2(d_2-1)(n-d_2)^2$ or all vertices in $\mathcal G_{u_1,u_2,v_1}$ have degree $d_1^4(n-d_2)^2n(n-1)$. Similar to the proof of Lemma \ref{lem:G_0coupling1},  $\mathfrak G$ is bipartite biregular. 
Therefore we have embedded $\mathfrak G_0$ into  $\mathfrak G$.\end{proof}

Now we are able to construct a coupling with the desired property.
In the graph $\mathfrak G$,
Choose a uniform element $X$ in $\mathcal G$ and consider $X^{(u_1u_2v_1)}$ to be the element in $\mathcal G_{u_1u_2v_1}$ given by walking from $X$ along an edge with probability proportional to its weight. Then $X^{(u_1u_2v_1)}$ is uniformly distributed in the vertex set $\mathcal G_{u_1u_2v_1}$.
Lemma \ref{lem:G_0coupling} yields a coupling of $(X, X^{(u_1u_2v_1)})$ that satisfies
\begin{align}
    &\mathbb P(X,X^{(u_1u_2v_1)} \text{ are identical or differ by a switching} \mid X^{(u_1u_2v_1)})\geq 1- \frac{2d_2}{n},\label{eq:lowerboundp111}\\
    & \mathbb P(X,X^{(u_1u_2v_1)} \text{ are identical or differ by a switching} \mid X)\geq 1- \frac{2d_2}{n-d_2}.\label{eq:lowerboundpp111}
\end{align}

\section{Concentration for linear functions of $XX^{\top}-d_1I$}\label{sec:highconcentrate}

 For a given $n\times n$ symmetric matrix $Q$, we define the following linear function for $X$:
\begin{align}
    g_{Q}(X)&:
    =\sum_{u_1,u_2\in [n]}Q_{u_1u_2}(XX^{\top}-d_1I)_{u_1u_2}=\sum_{\substack{u_1,u_2\in [n],u_1\not=u_2}} \sum_{v_1\in [m]}Q_{u_1u_2}  (X_{u_1v_1}X_{u_2v_1}).\label{eq:fQm}
\end{align}

In this section, we  analyze $g_Q(X)$ by the size biased coupling and the switching operations we introduced in Section \ref{sec:highswitching}.
 To quantify the change of $g_Q(X)$ after a valid forward switching, we first introduce the notion of \textit{codegrees}.
   
 \begin{definition}[codgeree]\label{def:codeg} Let $G$ be a bipartite graph with vertex sets $V_1=[n],V_2=[m]$ and  biadjacency matrix  $X$.
 Define the \textit{codegree} of two vertices $i,j\in [n]$ in $G$ as 
 \begin{align*}
     \textnormal{codeg}(X,i,j)=| \{v\in [m]: (i,v)\in E(G) \text{ and } (j,v)\in E(G)  \}|.
 \end{align*}
 Equivalently, we have
 $
     \textnormal{codeg}(X,i,j)=\sum_{v\in [m]}X_{iv}X_{jv}.
 $
 \end{definition}
 
 From Definition \ref{def:codeg}, Equation \eqref{eq:fQm} can also be written as 
 \begin{align}\label{eq:gcodeg}
     g_Q(X)=\sum_{\substack{u_1,u_2\in [n],u_1\not=u_2}} Q_{u_1u_2}\cdot \codeg(X,u_1,u_2).
 \end{align}
 For any $u_1,u_2\in [n],u_1\not=u_2, v_1\in [m]$, since $X$ is uniformly distributed, we have
$\mathbb EX_{u_1v_1}X_{u_2v_1}=\frac{d_1(d_2-1)}{n-1}.$
Denote 
\begin{align}\label{eq:defparameter}
    \mu:&=\mathbb Eg_Q(X)=\frac{d_1(d_2-1)}{n-1} \sum_{u_1\not=u_2}Q_{u_1u_2}, \quad  \tilde{\sigma}^2:=\frac{d_1(d_2-1)}{n-1}\sum_{u_1\not=u_2}Q_{u_1u_2}^2.
\end{align}

\begin{theorem} \label{thm:fQX}
 Let $X$ be the biadjacency matrix of a uniform random $(n,m,d_1,d_2)$-bipartite biregular graph and $M=XX^{\top}-d_1I$. Let $Q$ be an $n\times n$ symmetric matrix with all entries in $[0,a]$ and $Q_{ii}=0$ for $i\in [n]$. Let  $\tilde{\sigma}^2,\mu$ be the parameters defined in \eqref{eq:defparameter}. Denote  $p=1-\frac{2d_2}{n}$, $p'=1-\frac{2d_2}{n-d_2}$.
 Then for all $t\geq 0$,
 \begin{align}\label{eq:Xupper}
     &\mathbb P\left( g_Q(X)-\frac{\mu}{p}\geq t \right)\leq \exp\left(-\frac{\tilde{\sigma}^2}{6d_2pa^2}h\left( \frac{pat}{2\tilde{\sigma}^2}\right)\right),
     \end{align}
 and
     \begin{align}\label{eq:Xlower}
     &\mathbb P\left( g_Q(X)-p'\mu \leq - t \right)\leq \exp\left(-\frac{\tilde{\sigma}^2}{6d_2a^2}h\left( \frac{at}{2\tilde{\sigma}^2}\right)\right).
 \end{align}
\end{theorem}

\begin{proof}
 We now construct a size biased coupling for $g_Q(X)$ based the analysis of switchings in Section \ref{sec:highswitching}. Choose $X\in \mathcal G$ uniformly and walk through an edge with probability proportional to its weight. We then obtain a uniform random element $X^{(u_1u_2v_1)}$ in $\mathcal G_{u_1u_2v_1}$. The matrix $X^{(u_1u_2v_1)}$ is distributed as $X$ conditioned on $X_{u_1v_1}=X_{u_2v_1}=1.$ Independently of $X$, we choose $(U_1,U_2)=(u_1,u_2)$ with probability  
 \begin{align} \label{eq:PUV11}
 \mathbb P(U_1=u_1, U_2=u_2)=\frac{Q_{u_1u_2}}{\sum_{u\not=v}Q_{uv}}
 \end{align}for all $u_1\not=u_2$, and independently of everything  choose $V_1\in [m]$ uniformly. Set $X'=X^{(U_1,U_2,V_1)}$. By Lemma \ref{lem:biasindicator}, and \eqref{eq:fQm}, the pair $(g_{Q}(X),g_{Q}(X'))$ is a size biased coupling. 
 Let $\mathcal S_i(u_1,u_2,v_1)$ be the set of all tuples $(u_2,u_3,u_4,v_2,v_3)$ such that $(u_1,u_2,u_3,u_4,v_1,v_2,v_3)$ is a valid Type $i$ forward switching for $i=1,2$ and let $\mathcal S_3(u_1,u_2,v_1)$ be the set of all tuples $(u_2,\dots,u_6,v_2,\dots,v_5)$ such that $(u_1,\dots, u_6,v_1,\dots, v_5)$ is a valid Type $3$ forward switching.
Let $X^{(i)}(u_1,\dots,u_4,v_1,v_2,v_3)$ be the matrix obtained from  $X$ by a valid forward Type $i$ switching $(u_1,\dots,u_4,v_1,v_2,v_3)$ for $i=1,2$. Similarly, let $X^{(3)}(u_1,\dots,u_6,v_1,\dots,v_5)$ be the matrix obtained from  $X$ by a valid forward Type $3$ switching $(u_1,\dots,u_6,v_1,\dots,v_5)$.
From Lemma \ref{lem:G_0coupling}, the coupling of $(X,X')$ can be described as follows. 
Assuming the product $X_{U_1V_1}X_{U_2V_1}=0$ and conditioned on $X, U_1,U_2,V_1$, there is exactly one non-empty set among $\{\mathcal S_i(U_1,U_2,V_1)\}_{1\leq i\leq 3}$. The matrix $X'$ takes the value $X^{(i)}(u_1,u_2,u_3,u_4,v_1,v_2,v_3)$ with probability \[\frac{d_1^2d_2(n-d_2)}{d_1^4d_2(d_2-1)(n-d_2)^2}=\frac{1}{d_1^2(d_2-1)(n-d_2)}\] 
   for each $(u_1,\dots,u_4,v_1,v_2,v_3)\in \mathcal S_i(U_1,U_2,V_1)$ and $i=1,2$. And it takes the value $X^{(3)}(u_1,\dots,u_6,v_1,\dots,v_5)$ with probability $ \frac{1}{d_1^3d_2(d_2-1)(n-d_2)^2}$ for each $(u_1,\dots, u_6,v_1,\dots,v_5)\in \mathcal S_3(u_1,u_2,v_1)$. 
      For any valid Type 1 forward switching $(u_1,\dots,u_4,v_1,v_2,v_3)$, denote $\tilde{X}:=X^{(1)}(u_1,\dots,u_4,v_1,v_2,v_3)$. From \eqref{eq:gcodeg}, we have  
 \begin{align}  &g_Q(X^{(1)}(u_1,\dots,u_4,v_1,v_2,v_3))-g_Q(X)
 =\sum_{i\not=j, i,j\in [n]} Q_{ij} \left(\textnormal{codeg}(\tilde{X},i,j)-\textnormal{codeg}(X,i,j) \right) \notag\\
 \leq & \sum_{i\not=j, i,j\in [n]} Q_{ij} \left(\textnormal{codeg}(\tilde{X},i,j)-\textnormal{codeg}(X,i,j) \right)_+.\label{eq:gqplus}
 \end{align}
Then to have an upper bound on $g_Q(\tilde{X})-g_Q(X)$, it suffices to count the number of pairs $(i,j)$ with increased codegrees after applying the switching.
  From Figure \ref{fig:my_label}, the only new edges created in the valid Type 1 forward switching $(u_1,u_2,u_3,u_4,v_1,v_2,v_3)$ are $u_2v_1, u_3v_3$ and $u_4v_2$. When the edge $u_2v_1$ is added, for any $u\in \N(v_1),$ the codegree of $u$ and $u_2$ is increased by $1$. Similarly, when the edge $u_3v_3$ is added, for any $u\in \N(v_2)$, the codegree of $u$ and $u_4$ is increased by 1.  When the edge $u_4v_2$ is added, for any $u\in \N(v_3)$, the codegree of $u$ and $u_3$ is increased by 1. Therefore
\eqref{eq:gqplus} is bounded by 
 \begin{align}
&2\sum_{\substack{u\in \N(v_1)}} Q_{uu_2}+2\sum_{\substack{u\in \N(v_2)}} Q_{uu_4} +2\sum_{\substack{u\in \N(v_3)}} Q_{uu_3}\leq 6d_2a,\label{eq:g1}
 \end{align}
 where the factor $2$ comes from the symmetry of indices $i,j$ in \eqref{eq:gqplus}.  Using the same argument, for any valid Type 2 forward switching $(u_1,\dots,u_4,v_1,v_2,v_3)$,
  \begin{align}
g_Q(X^{(2)}(u_1,\dots,u_4,v_1,v_2,v_3))-g_Q(X) 
\leq  2\sum_{\substack{u\in \N(v_1) }} Q_{uu_1}+2\sum_{\substack{u\in \N(v_2) }} Q_{uu_4} +2\sum_{\substack{u\in \N(v_3) }} Q_{uu_3}\leq 6d_2a.\label{eq:g2}
 \end{align}
As can be seen from Figure \ref{fig:my_label12}, there are $6$ new edges created in a valid Type 3 forward switching. We also have for any valid Type 3 forward switching $(u_1,\dots,u_6,v_1,\dots,v_5)$,
  \begin{align}
 &g_Q(X^{(3)}(u_1,\dots,u_6,v_1,\dots,v_5))-g_Q(X) \notag\\
\leq &2\left(\sum_{\substack{u\in \N(v_1)}} Q_{uu_1}+\sum_{\substack{u\in \N(v_1)}} Q_{uu_2}+\sum_{\substack{u\in \N(v_2)}} Q_{uu_4} +\sum_{\substack{u\in \N(v_3)}} Q_{uu_3}+\sum_{\substack{u\in \N(v_4)}} Q_{uu_6}   +\sum_{\substack{u\in \N(v_5)}} Q_{uu_5}\right) \notag\\
\leq &  12d_2a.  \label{eq:g3}
 \end{align}

 If $X_{u_1v_1}=1, X_{u_2v_1}=0$, let $\overline{ \mathcal S_1}(u_1,u_2,v_1)$ be the set of tuples $(u_3,u_4,v_2,v_3)$  such that  $v_2\in \N(u_2)$, $u_3\in \N(v_1)$ with $u_3\not=u_1$, $u_4\in \overline{\N}(v_2)$, and  $v_3\in \N(u_4)$.  From the proof of Lemma \ref{lem:switchingcounts1}, 
there are $d_1^2(d_2-1)(n-d_2)$ many choices in total, hence 
$|\overline{ \mathcal  S_1}(u_1,u_2,v_1)|=d_1^2(d_2-1)(n-d_2)$ and we have $ \mathcal S_1(u_1,u_2,v_1) \subset \overline{\mathcal S_1}(u_1,u_2,v_1)$.  If the condition $X_{u_1v_1}=1, X_{u_2v_1}=0$ does not hold, set $\overline{ \mathcal S_1}(u_1,u_2,v_1)=\emptyset$.

If $X_{u_1v_1}=0, X_{u_2v_1}=1$, define $\overline{ \mathcal S_2}(u_1,u_2,v_1)$ to be the set of tuples $(u_3,u_4,v_2,v_3)$  such that  $v_2\in \N(u_1)$, $u_3\in \N(v_1)$ with $u_3\not=u_1$, $u_4\in \overline{\N}(v_2)$, and  $v_3\in \N(u_4)$.  We have  $|\overline{ \mathcal  S_2}(u_1,u_2,v_1)|=d_1^2(d_2-1)(n-d_2) $  
 and  $\mathcal S_2(u_1,u_2,v_1) \subset \overline{ \mathcal S_2}(u_1,u_2,v_1).$
If the condition  does not hold, set $\overline{ \mathcal S_2}(u_1,u_2,v_1)=\emptyset$.

If $X_{u_1v_1}=X_{u_2v_1}=0$, let $\overline{ \mathcal S_3}(u_1,u_2,v_1)$ be the set of tuples $(u_3,u_4,u_5,u_6,v_2,v_3,v_4)$  such that  $v_2\in \N(u_2)$, $u_3\in \N(v_1)$ with $u_3\not=u_1$, $u_4\in \overline{\N}(v_2)$,   $v_3\in \N(u_4)$, $u_6\in \overline{\N}(v_4), v_5\in \N(u_6)$.  It satisfies  $|\overline{\mathcal S_3}(u_1,u_2,v_1)|=d_1^4d_2(d_2-1)(n-d_2)^2 $
 and $\mathcal S_3(u_1,u_2,v_1) \subset \overline{\mathcal S_3}(u_1,u_2,v_1).$
If the condition does not hold, set $\overline{ \mathcal S_3}(u_1,u_2,v_1)=\emptyset.$

Denote $D:=(g_Q(X')-g_Q(X))_+$. We can write $\mathbb E[D\mathbf{1}_{\mathcal B}\mid X, U_1 ,U_2, V_1]$ as
 \begin{align}\label{eq:EDb}
 &\frac{1}{d_1^2(d_2-1)(n-d_2)}\sum_{(u_3,u_4,v_2,v_3)\in \mathcal S_1(u_1,u_2,v_1)}(g_Q(X^{(1)}(U_1, U_2,\dots,V_1,v_2,v_3))-g_Q(X))_+   \notag\\ 
     &+\frac{1}{d_1^2(d_2-1)(n-d_2)}\sum_{(u_3,u_4,v_2,v_3)\in \mathcal S_2(u_1,u_2,v_1)}(g_Q(X^{(2)}(U_1, U_2,\dots,V_1,v_2,v_3))-g_Q(X))_+  \notag\\
     &+\frac{1}{d_1^4d_2(d_2-1)(n-d_2)^2}\sum_{(u_3,\dots,u_6,v_2,v_3,v_4)\in \mathcal S_3(u_1,u_2,v_1)}(g_Q(X^{(3)}(U_1,U_2,\dots,V_1,\dots,v_5))-g_Q(X))_+ .  
\end{align}
Inequalities \eqref{eq:g1}, \eqref{eq:g2} and  \eqref{eq:g3} imply that \eqref{eq:EDb} is bounded by
\begin{align*}
      &\frac{2}{d_1^2(d_2-1)(n-d_2)}\sum_{(u_3,u_4,v_2,v_3)\in  \overline{\mathcal S_1}(U_1,U_2,V_1)}\left(\sum_{\substack{u\in \N(V_1)}} Q_{uU_2}+\sum_{\substack{u\in \N(v_2)}} Q_{uu_4} +\sum_{\substack{u\in \N(v_3)}} Q_{uu_3} \right)  \notag\\ 
     &+\frac{2}{d_1^2(d_2-1)(n-d_2)}\sum_{(u_3,u_4,v_2,v_3)\in  \overline{\mathcal S_2}(U_1,U_2,V_1)}\left(\sum_{\substack{u\in \N(V_1)}} Q_{uU_1}+
    \sum_{\substack{u\in \N(v_2)}} Q_{uu_4} +\sum_{\substack{u\in \N(v_3)}} Q_{uu_3}\right)\\
     &+\frac{2}{d_1^4d_2(d_2-1)(n-d_2)^2}
     \sum_{(u_3,\dots, u_6,v_2,v_3,v_4)\in  \overline{\mathcal S_3}(U_1,U_2,V_1)}\\
     &\left(\sum_{\substack{u\in \N(V_1)}} Q_{uU_1}+\sum_{\substack{u\in \N(V_1)}} Q_{uU_2}+\sum_{\substack{u\in \N(v_2)}} Q_{uu_4} +\sum_{\substack{u\in \N(v_3)}} Q_{uu_3}+\sum_{\substack{u\in \N(v_4)}} Q_{uu_6}   +\sum_{\substack{u\in \N(v_5)}} Q_{uu_5}\right).
 \end{align*}
 Recall the distribution of $U_1,U_2,V_1$ in \eqref{eq:PUV11},  and $\mu=\frac{d_1(d_2-1)}{n-1} \sum_{u_1\not=u_2}Q_{u_1u_2}.$ Taking expectation over $U_1,U_2,V_1$,  $\mathbb E[D\mathbf{1}_{\mathcal B}\mid X]$ is then upper bounded by the sum of the following three terms:
\begin{align}
\frac{2d_2}{\mu d_1^2n(n-1)(n-d_2)}\sum_{\substack{u_1\in [n]\\ v_1\in \N(u_1)\\ u_2\in \overline{\N}(v_1)}}\sum_{(u_3,u_4,v_2,v_3)\in  \overline{\mathcal S_1}(u_1,u_2,v_1)}Q_{u_1u_2}\left(\sum_{\substack{u\in \N(v_1)}} Q_{uu_2}+\sum_{\substack{u\in \N(v_2)}} Q_{uu_4} +\sum_{\substack{u\in \N(v_3)}} Q_{uu_3} \right), \label{eq:77}
    \end{align}

\begin{align}
 \frac{2d_2}{\mu d_1^2n(n-1)(n-d_2)}\sum_{\substack{u_1\in [n]\\ v_1\in \overline{\N}(u_1)\\ u_2\in \N(v_1)}}\sum_{(u_3,u_4,v_2,v_3)\in \overline {\mathcal S_2}(u_1,u_2,v_1)}Q_{u_1u_2}\left(\sum_{\substack{u\in \N(v_1)}} Q_{uu_1}+
    \sum_{\substack{u\in \N(v_2)}} Q_{uu_4} +\sum_{\substack{u\in \N(v_3)}} Q_{uu_3}\right),\label{eq:78}
    \end{align}
    and
\begin{align}
    & \frac{2}{\mu d_1^4 n(n-1)(n-d_2)^2}\sum_{\substack{u_1\in [n], v_1\in \overline{\N}(u_1), u_2\in \overline{\N}(v_1)}}
     \sum_{(u_3,\dots,v_4)\in \overline {\mathcal S_3}(u_1,u_2,v_1)} \label{eq:79} \\
     &Q_{u_1u_2}\left(\sum_{\substack{u\in \N(v_1)}} Q_{uu_1}+\sum_{\substack{u\in \N(v_1)}} Q_{uu_2}+\sum_{\substack{u\in \N(v_2)}} Q_{uu_4} +\sum_{\substack{u\in \N(v_3)}} Q_{uu_3}+\sum_{\substack{u\in \N(v_4)}} Q_{uu_6}   +\sum_{\substack{u\in \N(v_5)}} Q_{uu_5}\right) \notag\\
     &:=S_{31}+S_{32}+S_{33}+S_{34}+S_{35}+ S_{36}. \notag
\end{align}
In the following proof, we estimate the three terms \eqref{eq:77}, \eqref{eq:78} and \eqref{eq:79} separately.  Write the sum in \eqref{eq:77} as
\begin{align}\label{eq:Q1Q2Q3}
    &\sum_{\substack{u_1\in [n]\\ v_1\in \N(u_1), u_2\in \overline{\N}(v_1)}}\sum_{(u_3,u_4,v_2,v_3)\in  \overline{\mathcal S}_1(u_1,u_2,v_1)}\left(\sum_{\substack{u\in \N(v_1)}} Q_{u_1u_2}Q_{uu_2}+\sum_{\substack{u\in \N(v_2)}} Q_{u_1u_2}Q_{uu_4} +\sum_{\substack{u\in \N(v_3)}} Q_{u_1u_2}Q_{uu_3} \right)\\
    &:=S_{11}+S_{12}+S_{13}. \notag
\end{align}
By Cauchy's inequality, $S_{11}$ is bounded by
\begin{align}
    \left(\sum_{\substack{u_1\in [n]\\ v_1\in \N(u_1)\\ u_2\in \overline{\N}(v_1)}}\sum_{(u_3,u_4,v_2,v_3)\in  \overline{\mathcal S}_1(u_1,u_2,v_1)}\sum_{\substack{u\in \N(v_1) }} Q_{u_1u_2}^2\right)^{1/2}\left(\sum_{\substack{u_1\in [n]\\ v_1\in \N(u_1)\\ u_2\in \overline{\N}(v_1)}}\sum_{(u_3,u_4,v_2,v_3)\in  \overline{\mathcal S}_1(u_1,u_2,v_1)}\sum_{\substack{u\in \N(v_1)}} Q_{uu_2}^2\right)^{1/2}. \label{eq:QQQQQQ}
\end{align}
Since
$|\overline{ \mathcal  S_1}(u_1,u_2,v_1)|=d_1^2(d_2-1)(n-d_2)$ and $\tilde{\sigma}^2=\frac{d_1(d_2-1)}{n-1}\sum_{u_1\not=u_2}Q_{u_1u_2}^2$, the first factor in the product of \eqref{eq:QQQQQQ} satisfies
\begin{align}
    &\sum_{\substack{u_1\in [n]\\ v_1\in \N(u_1), u_2\in \overline{\N}(v_1)}}\sum_{(u_3,u_4,v_2,v_3)\in  \overline{\mathcal S}_1(u_1,u_2,v_1)}\sum_{\substack{u\in \N(v_1)}} Q_{u_1u_2}^2 
    =  d_1^2d_2(d_2-1)(n-d_2)\sum_{\substack{u_1\in [n]\\ v_1\in \N(u_1), u_2\in \overline{\N}(v_1)}}  Q_{u_1u_2}^2.  \label{eq:S1X}
    \end{align}
  Since
    \[\sum_{\substack{u_1\in [n], v_1\in \N(u_1), u_2\in \overline{\N}(v_1)}}  Q_{u_1u_2}^2\leq \sum_{u_1,u_2\in [n],v_1\in \N(u_1)}Q_{u_1u_2}^2=d_1\sum_{u_1\not=u_2}Q_{u_1u_2}^2, \]
\eqref{eq:S1X} is then bounded by    
\begin{align}\label{eq:SIXX}
d_1^3d_2(d_2-1)(n-d_2) \sum_{u_1\not=u_2}Q_{u_1u_2}^2
    =d_1^2d_2(n-1)(n-d_2)\tilde{\sigma}^2.
    \end{align}
We also have
\begin{align*}
 \sum_{\substack{u_1\in [n], v_1\in \N(u_1), u_2\in \overline{\N}(v_1)}}\sum_{\substack{u\in \N(v_1)}} Q_{uu_2}^2&
 \leq \sum_{u_2\in [n]} \sum_{\substack{u_1\in [n] ,
 v_1\in \N(u_1)}}\sum_{u\in \N(v_1)} Q_{uu_2}^2 \leq  d_1d_2\sum_{u\not=u_2}Q_{uu_2}^2=\frac{d_2(n-1)\tilde{\sigma}^2}{d_2-1},  
\end{align*}
where the second inequality comes from the fact that each $u\in [n]$ is counted $d_1d_2$ times in the sum. Then the second factor in \eqref{eq:QQQQQQ} is bounded by 
\begin{align}
 d_1^2(d_2-1)(n-d_2)\cdot \frac{d_2(n-1)\tilde{\sigma}^2}{d_2-1} 
    &= d_1^2d_2(n-1)(n-d_2)\tilde{\sigma}^2.\label{eq:notethat}
\end{align}
From \eqref{eq:QQQQQQ}, \eqref{eq:SIXX} and \eqref{eq:notethat},  we have 
\begin{align}
  &\frac{2d_2}{\mu d_1^2n(n-1)(n-d_2)} S_{11}
  \leq  \frac{2d_2}{\mu d_1^2n(n-1)(n-d_2)} \cdot  d_1^2d_2(n-1)(n-d_2)\tilde{\sigma}^2 =  \frac{2\tilde{\sigma}^2}{\mu} \cdot \frac{d_2^2}{n}. \label{eq:newbound1}
\end{align} Similarly,  by Cauchy's inequality, 
\begin{align}
   S_{12}
    \leq &[d_1^2d_2(n-1)(n-d_2)\tilde{\sigma}^2]^{1/2} \left( \sum_{\substack{u_1\in [n], v_1\in \N(u_1), u_2\in \overline{\N}(v_1)}}\sum_{(u_3,u_4,v_2,v_3)\in  \overline{\mathcal S}_1(u_1,u_2,v_1)}~\sum_{\substack{u\in \N(v_2)}} Q_{uu_4}^2\right)^{1/2}. \notag
\end{align}
For a given $(v_2,u_4)$, by taking $v_3\in \N(u_4),u_2\in \N(v_2), v_1\in \overline{\N}(u_2), u_1\in \N(v_1), u_3\in \N(v_1), u_3\not=u_1$, there are at most 
$(m-d_1)d_1d_2^2(d_2-1)$ many tuples $(u_1,u_2,u_3,u_4,v_1,v_3)$ such that $(u_1,u_2,u_3,u_4,v_1,v_2, v_3)\in  \overline{\mathcal S_1}(u_1,u_2,v_1).$
Hence 
\begin{align*}
  & \sum_{\substack{u_1\in [n]\\v_1\in \N(u_1), u_2\in \overline{\N}(v_1)}}\sum_{(u_3,u_4,v_2,v_3)\in  \overline{\mathcal S}_1(u_1,u_2,v_1)}~\sum_{\substack{u\in \N(v_2)}} Q_{uu_4}^2
  \leq (m-d_1)d_1 d_2^2(d_2-1) \sum_{v_2\in [m], u_4\in [n], u\in \N(v_2)} Q_{uu_4}^2\\
  \leq & (m-d_1)d_1^2d_2^2(d_2-1)\sum_{u\not=u_4}Q_{uu_4}^2
  =(m-d_1)(n-1)d_1d_2^2\tilde{\sigma}^2=d_1^2d_2(n-1)(n-d_2)\tilde{\sigma}^2.
\end{align*}
Then 
\begin{align}
  & \frac{2d_2}{\mu d_1^2n(n-1)(n-d_2)} S_{12}
\leq  \frac{2\tilde{\sigma}^2}{\mu} \cdot \frac{d_2^2}{n}.\label{eq:newbound2}
\end{align}
For the third term in \eqref{eq:Q1Q2Q3}, we have 
\begin{align}
   S_{13}
    \leq &[d_1^2(d_2-1)(n-1)(n-d_2)\tilde{\sigma}^2]^{1/2} \left( \sum_{\substack{u_1\in [n]\\ v_1\in \N(u_1), u_2\in \overline{\N}(v_1)}}\sum_{(u_3,u_4,v_2,v_3)\in  \overline{\mathcal S}_1(u_1,u_2,v_1)}~\sum_{\substack{u\in \N(v_3)}} Q_{uu_3}^2\right)^{1/2}. \notag
\end{align}
For fixed $v_3,u_3$, by choosing $v_1\in \N(u_3), u_1\in \N(v_1), u_1\not=u_3, u_2\in \overline{\N}(v_1),v_2\in \N(u_2),u_4\in \N(u_3),$ there are at most $d_1^2d_2(d_2-1)(n-d_2)$ many tuples $(u_1,u_2,u_3,u_4,v_1,v_2, v_3)$ such that $(u_1,u_2,u_3,u_4,v_1,v_2, v_3)\in  \overline{\mathcal S}_1(u_1,u_2,v_1).$
Therefore by the same argument, we have
\begin{align*}
  &\sum_{\substack{u_1\in [n], v_1\in \N(u_1), u_2\in \overline{\N}(v_1)}}\sum_{(u_3,u_4,v_2,v_3)\in \overline{  S}_1(u_1,u_2,v_1)} \sum_{u\in \N(v_3)} Q_{uu_3}^2
  \leq d_1^2d_2(n-d_2)(n-1)\tilde{\sigma}^2,
\end{align*}
and 
\begin{align}
   \frac{2d_2}{\mu d_1^2n(n-1)(n-d_2)} S_{13}\leq \frac{2\tilde{\sigma}^2}{\mu} \cdot \frac{d_2^2}{n}. \label{eq:newbound3}
\end{align}
From \eqref{eq:newbound1}, \eqref{eq:newbound2} and \eqref{eq:newbound3}, the term \eqref{eq:77} is bounded by
$ \frac{6\tilde{\sigma}_2^2d_2^2}{\mu n}$.
 The bound on \eqref{eq:77} also holds for \eqref{eq:78} by the symmetric role of the two vertices $u_1,u_2$.

  Now it remains to estimate \eqref{eq:79}.  
Recall  $| \overline{\mathcal S}_3(u_1,u_2,v_1)|=d_1^4d_2(d_2-1)(n-d_2)^2$. By Cauchy's inequality, 
\begin{align}\label{eq:Quu11}
S_{31}\leq &\left(\sum_{\substack{u_1\not=u_2\in [n]\\ v_1\in [m]}}\sum_{(u_3,\dots,v_4)\in \overline{ \mathcal  S_3}(u_1,u_2,v_1)}\sum_{u\in \N(v_1)}Q_{u_1u_2}^2\right)^{1/2}\left(\sum_{\substack{u_1\not=u_2\in [n]\\ v_1\in [m]}}\sum_{(u_3,\dots,v_4)\in \overline{  \mathcal S_3}(u_1,u_2,v_1)}\sum_{\substack{u\in \N(v_1)}}Q_{uu_1}^2\right)^{1/2} \\
=&[d_1^4d_2(n-d_2)^2n(n-1)\tilde{\sigma}^2]^{1/2}\left(\sum_{\substack{u_1\not=u_2\in [n], v_1\in [m]}}\sum_{(u_3,\dots,v_4)\in \overline{ \mathcal S_3}(u_1,u_2,v_1)}\sum_{\substack{u\in \N(v_1)}}Q_{uu_1}^2\right)^{1/2}.\notag
\end{align}
  For fixed $v_1,u_1$, by taking $u_2\in \overline{\N}(v_1), u_3\in \N(v_1), v_2\in \N(u_2), u_4\in \overline{\N}(v_2), v_3\in \N(u_4),u_5\in \N(v_1),u_5\not=u_3, v_4\in \N(u_1), u_6\in \overline{\N}(v_4)$ and $v_5\in \N(u_6),$ there are at least $d_1^4d_2(d_2-1)(n-d_2)^3$ many choices of $(u_2,u_3,u_4,u_5,u_6,v_2,v_3,v_4,v_5)$ such that
 $(u_3,u_4,u_5,u_6,v_2,v_3,v_4,v_5)\in \overline{ \mathcal S_3}(u_1,u_2,v_1). $
  Then
 \begin{align}
    & \sum_{\substack{u_1\not=u_2\in [n], v_1\in [m]}}\sum_{(u_3,\dots,v_4)\in \overline{  \mathcal S_3}(u_1,u_2,v_1)}\sum_{\substack{u\in \N(v_1)}}Q_{uu_1}^2  \notag\\
     \leq &  d_1^4d_2(d_2-1)(n-d_2)^3\sum_{u_1\in [n],v_1\in [m]}\sum_{u\in \N (v_1)}Q_{uu_1}^2    \notag \\
     =& d_1^5d_2(d_2-1)(n-d_2)^3\sum_{u\not=u_1}Q_{uu_1}^2 
     =d_1^4d_2(n-d_2)^3(n-1)\tilde{\sigma}^2.\label{eq:Quu1}
 \end{align}
 Hence  
 \begin{align}
     \frac{2}{\mu d_1^4n(n-1)(n-d_2)^2} S_{31}
     \leq &\frac{2\tilde{\sigma}^2}{\mu}\frac{d_2\sqrt{n-d_2}}{\sqrt n}\leq \frac{2\tilde{\sigma}^2d_2}{\mu}.\label{eq:checkother}
 \end{align}
 Similarly, by the symmetry of $u_1$ and $u_2$, we obtain
$
     \frac{2}{\mu d_1^4n(n-1)(n-d_2)^2}S_{32}
     \leq \frac{2\tilde{\sigma}^2d_2}{\mu}.
$
   Using the similar argument for the proof of \eqref{eq:Quu1}, 
 the upper bound in \eqref{eq:checkother} holds for all other $4$ terms in \eqref{eq:79} as well. 
Therefore \eqref{eq:79} is bounded by
$ \frac{12\tilde{\sigma}^2d_2}{\mu}.$
Combining the estimates for \eqref{eq:77}, \eqref{eq:78} and \eqref{eq:79}, since $d_2\leq n$, we then obtain
 \begin{align}\label{eq:rmk}
     \mathbb E[D \mathbf{1}_{\mathcal B}\mid X]\leq \frac{12\tilde{\sigma}^2d_2^2}{\mu n} +\frac{12\tilde{\sigma}^2d_2}{\mu}\leq \frac{24\tilde{\sigma}^2d_2}{\mu}.
 \end{align}
By taking $\tau^2=24\tilde{\sigma}^2d_2, c=12d_2a$ in Theorem \ref{thm:tailestimate1}, the result follows.
\end{proof}

Similar to Corollary \ref{cor:concentration}, we obtain the following concentration inequalities from Theorem \ref{thm:fQX}.
\begin{corollary}\label{cor:cornew}
Let $\gamma_0=\frac{1}{p}-1=\frac{2d_2}{n-2d_2}, c_0=\frac{p}{6}=\frac{1}{6}(1-\frac{2d_2}{n}).$ We have 
\begin{align}
  \mathbb P (g_Q(X)\geq (1+\gamma_0)\mu+t)&\leq \exp\left(-c_0\frac{\tilde{\sigma}^2}{d_2a^2}h\left( \frac{at}{2\tilde{\sigma}^2}\right)\right),\label{eq:ineq1} \\
  \mathbb P (|g_Q(X)-\mu|\geq \gamma_0\mu+t)& \leq 2\exp \left(-\frac{c_0t^2}{8d_2(\tilde{\sigma}^2+\frac{1}{6}at)}\right). \label{eq:ineq3} 
\end{align}
\end{corollary}

\section{Proof of Theorem \ref{thm:main2}}\label{sec:finalsec}
 The remaining analysis is similar to the analysis in  Section \ref{sec:KSargument} for the proof of Theorem \ref{thm:main1}. We will only address the main differences.  Let $M=XX^T-d_1 I$. For convenience we now write \eqref{eq:fQm} as
\[g_Q(X)=f_Q(M)=\sum_{u_1,u_2\in [n]}Q_{u_1u_2}M_{u_1u_2}.\]
Recall
$
    \lambda(M)=\sup_{x\in S_0^{n-1}} |\langle x,Mx\rangle |.$
For fixed $x\in S^{n-1}_0$, we split the sum into light and heavy parts. Define \textit{light and heavy couples} by
\begin{align*}
    \mathcal L(x)&=\{u,v\in [n]: |x_ux_v|\leq \sqrt{d_1(d_2-1)}/n \}, \quad 
     \mathcal H(x)=\{u,v\in [n]: |x_ux_v|> \sqrt{d_1(d_2-1)}/n \}.
\end{align*}
We can decompose the linear form $f_Q(M)$ as 
\[ f_{xx^{\top}}(M)=\langle x, Mx\rangle =f_{\mathcal L(x)}(M)+f_{\mathcal H(x)}(M),\]
where 
\begin{align}
 &f_{\mathcal L(x)}(M)=\sum_{(u,v)\in \mathcal L(x)}x_ux_vM_{uv},\quad f_{\mathcal H(x)}(M)=\sum_{(u,v)\in \mathcal H(x)}x_ux_vM_{uv}.  
\end{align}

\begin{lemma}\label{eq:b+3}
For any fixed $x\in S_0^{n-1}$,  $\beta \geq 4\gamma_0\sqrt{d_1(d_2-1)}$, and $n\geq 2$,
\begin{align}\label{eq:lightinequality}
    \mathbb P\left(|f_{\mathcal L(x)}(M)|\geq (\beta+3)\sqrt{d_1(d_2-1)}\right)\leq 4 \exp \left( -\frac{3c_0\beta^2n}{8d_2(24+\beta)}\right).
\end{align}
\end{lemma}
\begin{proof}
Recall for $u\not=v$, $\mathbb EM_{uv}=\frac{d_1(d_2-1)}{n-1}.$ For any fixed $x\in S_0^{n-1}$,
\begin{align*}
    |\mathbb Ef_{\mathcal L(x)}(M)|&\leq |\mathbb E\langle x,Mx\rangle |+|\mathbb Ef_{\mathcal H(x)}(M)|\\
    &\leq \left|x^{\top}\left(\mathbb EM-\frac{d_1(d_2-1)}{n}\mathbf{1}_n\mathbf{1}_n^{\top}\right) x\right|+\frac{d_1(d_2-1)}{n-1}\sum_{(u,v)\in \mathcal H(x)}|x_ux_v|\\
    &\leq \left\| \mathbb EM-\frac{d_1(d_2-1)}{n}\mathbf{1}_n\mathbf{1}_n^{\top}\right\|_{F}+\frac{d_1(d_2-1)}{n-1}\sum_{(u,v)\in \mathcal H(x)}\frac{|x_ux_v|^2}{\sqrt{d_1(d_2-1)}/n}\\
    &\leq \frac{d_1(d_2-1)}{\sqrt{n-1}}+\sqrt{d_1(d_2-1)}\frac{n}{n-1}\leq 3\sqrt{d_1(d_2-1)}.\label{eq:twopart2}
\end{align*}
We split the set $\mathcal L(x)$  into two parts as  $\mathcal L(x)=\mathcal L_+(x)\cup \mathcal L_-(x)$ where  
\[\mathcal L_+(x)=\{u,v\in [n]: 0\leq x_ux_v\leq \sqrt{d_1(d_2-1)}/n \}, \quad \mathcal L_-(x)=\mathcal L\setminus \mathcal L_+(x).\]
Then 
\begin{align}
      | f_{\mathcal L(x)}(M)-\mathbb Ef_{\mathcal L(x)}(M)| 
     \leq & | f_{\mathcal L_+(x)}(M)-\mathbb Ef_{\mathcal L_+(x)}(M)|+| f_{\mathcal L_-(x)}(M)-\mathbb Ef_{\mathcal L_-(x)}(M)|.\label{eq:splitterm1}
\end{align}
Consider the first term in the right hand side of \eqref{eq:splitterm1}. By Cauchy's inequality,
\begin{align*}
    \mu=\mathbb Ef_{\mathcal L_+(x)}(M)\leq \frac{d_1(d_2-1)}{n-1}\sum_{uv}|x_ux_v|\leq d_1(d_2-1)\frac{n}{n-1}\leq 2d_1(d_2-1).  
\end{align*}
Also we have
$
        \tilde{\sigma}^2=\sum_{uv\in \mathcal L_+(x)}|x_ux_v|^2\mathbb EM_{uv}\leq \frac{d_1(d_2-1)}{n-1}.
$
Then by \eqref{eq:ineq3} with $a=\frac{\sqrt{d_1(d_2-1)}}{n}$, for any $\beta>4\gamma_0\sqrt{d_1(d_2-1)},$  we have
\begin{align*}
    &\mathbb P\left( | f_{\mathcal L_+(x)}(M)|-\mathbb Ef_{\mathcal L_+(x)}(M)|\geq (\beta/2)\sqrt{d_1(d_2-1)}\right)
\leq  2\exp \left( -\frac{3c_0\beta^2n}{8d_2(24+\beta)}\right).
\end{align*}
The same bound holds for the second term in \eqref{eq:splitterm1}.
Then  with probability at least $1-4\exp \left( -\frac{3c_0\beta^2n}{8d_2(24+\beta)}\right)$,  
\begin{align*}
    |f_{\mathcal L(x)}(M)|&\leq |\mathbb Ef_{\mathcal L(x)}(M)|+ | f_{\mathcal L(x)}(M)-\mathbb Ef_{\mathcal L(x)}(M)|\leq 
    (3+\beta)\sqrt{d_1(d_2-1)}.
\end{align*}
This completes the proof.
\end{proof}

Similar to Lemma \ref{lem:DPproperty},  the following lemma shows that the discrepancy property for $M=XX^{\top}-d_1I$ holds for a random bipartite biregular graph with high probability. The proof is similar, and we skip the details.
\begin{lemma}\label{lem:DPproperty11}
Let $M=XX^{\top}-d_1I$ where $X$ is the biadjacency matrix of a uniform random $(n,m,d_1,d_2)$-bipartite biregular graph. For any $K\geq 0$,
With probability at least $1-n^{-K}$, $\textnormal{DP}\left(\delta,\kappa_1,\kappa_2\right)$ holds for $M$ with $\delta=\frac{d_1(d_2-1)}{n-1}$,
$\kappa_1=e^2(1+\gamma_0)^2$,and $ \kappa_2=\frac{8d_2}{c_0}(1+\gamma_0)(K+4)$.
\end{lemma}

\begin{lemma}\label{lem:a0}
Let $M=XX^{\top}-d_1I$ where $X$ is  the biadjacency matrix of a $(n,m,d_1,d_2)$-bipartite biregular graph. Suppose $M$ has $\textnormal{DP}(\delta,\kappa_1,\kappa_2)$ with $\delta,\kappa_1,\kappa_2$ given in Lemma \ref{lem:DPproperty11}. Then there exists a constant $\alpha_0$ depending on $\kappa_1,\kappa_2$ such that
$ f_{\mathcal H(x)}(M)\leq \alpha_0\sqrt{d_1(d_2-1)},$
where  $\alpha_0=16+64(\kappa_1+1)+64\kappa_2\left(1+\frac{2}{\kappa_1\log\kappa_1}\right).$
\end{lemma}
\begin{proof} Note that $\delta=\frac{d_1(d_2-1)}{n-1}\leq \frac{2d_1(d_2-1)}{n}$ for $n\geq 2$.
The proof follows verbatim from  \cite[Lemma 6.6]{cook2018size}.
\end{proof}

Now we are ready to prove Theorem \ref{thm:main2} using the $\epsilon$-net argument.

\begin{lemma}
For $\epsilon\in (0,1/2)$, let  $\N_{\epsilon}^0$ be an $\epsilon$-net of $S_0^{n-1}$. Let $X$ be the biadjacency matrix of a $(n,m,d_1,d_2)$-bipartite biregular graph and $M=XX^{\top}-d_1I$. Then
\begin{align}\label{eq:lambdaM}
   \lambda(M)\leq \frac{1}{1-2\epsilon}\sup_{ x\in \N_{\epsilon}^0} |\langle x, Mx\rangle |.
\end{align}
\end{lemma}

\begin{proof}[Proof of Theorem \ref{thm:main2}]

 Fix $K>0$. By Lemma \ref{lem:DPproperty11}, with probability at least $1-n^{-K}$, $M$ has $\textnormal{DP}(\delta,\kappa_1,\kappa_2)$ property where the parameters $\delta,\kappa_1,\kappa_2$ are given in Lemma \ref{lem:DPproperty11}. Let $\mathcal D$ be the event that this property holds. Then it suffices to show 
 \begin{align*}
     \mathbb P\left(\mathcal D\cap \left\{ \lambda(M)\geq \alpha\sqrt{d_1(d_2-1)}\right\}\right)&\leq e^{-n}.
 \end{align*}
 Take $\epsilon=1/4$ in \eqref{eq:lambdaM}. Then
 $
     \lambda(M)\leq 2\sup_{x\in \N_{\epsilon}^0} |\langle x, Mx\rangle |.
$
 We obtain
 \begin{align}\label{eq:obtain}
     \mathbb P\left(\mathcal D\cap \left\{ \lambda(M)\geq \alpha\sqrt{d_1(d_2-1)}\right\}\right)&\leq \sum_{x\in \N_{\epsilon}^0}\mathbb P\left(\mathcal D\cap \left\{ |\langle x,Mx\rangle |\geq (\alpha/2)\sqrt{d_1(d_1-1)}\right\}\right).
 \end{align}
 For any fixed $x\in \N_{\epsilon}$,
 \begin{align*}
 &\mathbb P\left(\mathcal D\cap \left\{ |\langle x,Mx\rangle |\geq (\alpha/2)\sqrt{d_1(d_2-1)}\right\}\right)\\
 &\leq  \mathbb P\left(\mathcal D\cap \left\{ |f_{\mathcal L(x)}(M)|\geq (\alpha/2)\sqrt{d_1(d_2-1)}-|f_{\mathcal H(x)}(M)|\right\}\right)\\
 &\leq \mathbb P\left( |f_{\mathcal L(x)}(M)|\geq (\alpha/2-\alpha_0)\sqrt{d_1(d_2-1)} \right).
 \end{align*}
 Take $\beta=\frac{1}{2}\alpha-\alpha_0-3$. When $\beta>4\gamma_0\sqrt{d_2(d_1-1)}$,  from \eqref{eq:lightinequality}, we have 
 \begin{align}\label{eq:true}
\mathbb P\left( \mathcal D\cap \left\{ |\langle x,Mx\rangle |\geq (\alpha/2)\sqrt{d_1(d_2-1)}\right\}\right)\leq 2\exp\left(-\frac{3c_0\beta^2n}{8d_2(24+\beta)}\right).
 \end{align}
 Recall $\gamma_0=\frac{2d_2}{n-2d_2}, c_0=\frac{1}{6}(1-\frac{2d_2}{n})$ in Corollary \ref{cor:cornew}. In  the  assumption of Theorem \ref{thm:main2} we also have $d_2\leq n/4$  and $ d_1\leq C_1n^2.$  Then it follows that $c_0\geq \frac{1}{12}, \gamma_0\leq 1$ and 
\begin{align}\label{eq:betta}
& 4\gamma_0 \sqrt{d_1(d_2-1)}=\frac{8d_2\sqrt{d_1(d_2-1)}}{n-2d_2}\leq \frac{8d_2\sqrt{C_1(d_2-1)}n}{n-2d_2}\leq 16d_2\sqrt{(d_2-1)C_1}.    
\end{align}
Also recall $\kappa_1=e^2(1+\gamma_0)^2$, $ \kappa_2=\frac{8d_2}{c_0}(1+\gamma_0)(K+4)$ from Lemma \ref{lem:DPproperty11}. We have 
\[e^2\leq \kappa_1\leq 4e^2,  \quad  \kappa_2\leq 192d_2(K+4) .\]
Then $\alpha_0$ given in Lemma \ref{lem:a0}
is a bounded constant depending on $d_2$ and $K$. 
Since $|\mathcal N_{1/4}^0|\leq 9^n$, from \eqref{eq:obtain} we have 
 \begin{align}
     \mathbb P\left(\mathcal D\cap \left\{ \lambda(M)\geq \alpha\sqrt{d_1(d_2-1)}\right\}\right)&\leq 2\cdot 9^{n}\exp\left(-\frac{3c_0\beta^2n}{8d_2(24+\beta)}\right) .\label{eq:818}
 \end{align}
 Taking $\beta>16d_2\sqrt{(d_2-1)C_1}$, then from \eqref{eq:betta}, the condition  $\beta>4\gamma_0\sqrt{d_2(d_2-1)}$ holds for \eqref{eq:true}. We obtain from  \eqref{eq:818} that for sufficiently large $\beta$ depending on $d_2,C_1$,
 \[ \mathbb P\left(\mathcal D\cap \left\{ \lambda(M)\geq \alpha\sqrt{d_1(d_2-1)}\right\}\right)\leq 2\cdot 9^n \exp\left(-\frac{\frac{1}{4}\beta^2n}{8d_2(24+\beta)}\right)\leq e^{-n},\]
where $\alpha=2(\beta+\alpha_0+3)$. We have 
 $
   \mathbb P\left( \lambda(M)\geq \alpha\sqrt{d_1(d_2-1)}\right)  \leq n^{-K}+e^{-n}.
$ 
This completes the proof by taking a new constant $\alpha'=\alpha\sqrt{d_2-1}$.\end{proof}

 \begin{remark}\label{rmk:final}
  In \eqref{eq:818}, we see that if $d_2$ is not a bounded constant, the probability bound cannot be $o(1)$, since the union bound is taken over  exponentially many points on a $\epsilon$-net. This is a limitation of the method we use and we must assume $d_2$ is fixed in our proof. 
 \end{remark}

\section*{Acknowledgements}
The author thanks Ioana Dumitriu and  Tobias Johnson for helpful discussion. This work is partially supported by NSF DMS-1949617. The author acknowledges
support from NSF DMS-1928930 during his participation in the program “Universality and Integrability
in Random Matrix Theory and Interacting Particle Systems” hosted by the Mathematical Sciences Research
Institute in Berkeley, California during the Fall semester of 2021.

\bibliographystyle{plain}
\bibliography{jotp.bib}
\end{document}